\documentclass[preprint,12pt]{article}
\usepackage[top=1.1in, bottom=1.2in, left=0.8in, right=0.8in]{geometry}
\usepackage{amsmath,amsfonts,amssymb,amsthm}
\usepackage{mathrsfs,dsfont}
\usepackage{graphicx}
\usepackage{colortbl,dcolumn}
\usepackage{psfrag}
\usepackage{booktabs}
\usepackage{cite}
\usepackage{url}
\usepackage{comment}
\usepackage{hyperref}
\usepackage{subfigure}
\allowdisplaybreaks[4]
\numberwithin{equation}{section}

\newcommand{\R}{\mathbb{R}}
\newcommand{\N}{\mathbb{N}}
\newcommand{\E}{\mathbb{E}}
\renewcommand{\P}{\mathbb{P}}

\newcommand{\F}{\mathcal{F}}

\newtheorem{theorem}{Theorem}[section]

\newtheorem{lemma}[theorem]{Lemma}
\newtheorem{proposition}[theorem]{Proposition}

\newtheorem{remark}[theorem]{Remark}
\newtheorem{example}[theorem]{Example}
\newtheorem{assumption}[theorem]{Assumption}
\newcounter{algorithm}

\begin{document}


\title{Weak order one convergence of structure-preserving stochastic theta methods for stochastic differential algebraic equations with time-dependent singular matrices\footnotemark[1]}

\footnotetext{\footnotemark[1] Z. Chen is supported by Yunnan Fundamental Research Project (No. 202601AT070161) and  National Natural Science Foundation of China (No. 12201552). L. Chen is supported by National Natural Science Foundation of China (No. 11961029) and Jiangxi Provincial Natural Science Foundation (No. 20242BAB23004).}

\author{Caiyuan Zhu\footnotemark[2],\quad
Ziheng Chen\footnotemark[3],\quad
Lin Chen\footnotemark[4],\quad
Yiwei Zhou\footnotemark[5]}

\footnotetext{\footnotemark[2] School of Mathematics and Statistics, Yunnan University, Kunming, Yunnan, 650500, China. Email: 12024113103@stu.ynu.edu.cn}

\footnotetext{\footnotemark[3] School of Mathematics and Statistics, Yunnan University, Kunming, Yunnan, 650500, China. Email: czh@ynu.edu.cn. Corresponding author}

\footnotetext{\footnotemark[4] School of Statistics and Data Science, Jiangxi University of Finance and Economics, Nanchang, 330013, China. Email: chenlin@jxufe.edu.cn}

\footnotetext{\footnotemark[5] School of Mathematics and Statistics, Yunnan University, Kunming, Yunnan, 650500, China. Email: yiwei.zhou@utexas.edu}

\maketitle

\begin{abstract}
      {\rm\small This paper studies the weak convergence order of structure-preserving stochastic theta methods for a class of index-$1$ stochastic differential algebraic equations with time-dependent singular matrices. The singular matrix is allowed to vary in time but preserves a fixed differential-algebraic splitting, thereby extending the constant singular-matrix setting while retaining the projector structure required for constraint preservation. By exploiting the index-$1$ algebraic-differential decomposition of the exact solution, we establish an abstract weak convergence theorem for constraint-preserving one-step approximations and apply it to the stochastic theta method with $\theta \in (0,1]$. Under global Lipschitz, linear growth, and suitable smoothness assumptions, the considered method is proved to be well posed, to preserve the algebraic constraints at all time levels, and to converge with weak order one. Numerical experiments are finally presented to confirm the structure-preserving property and the theoretical convergence order.} \\


\end{abstract}

\section{Introduction}

Stochastic differential equations (SDEs) provide a fundamental mathematical framework for modelling random phenomena arising in physics, engineering, biology, finance, and many other areas. However, the state variables in many stochastic systems do not evolve freely in the ambient state space, but are restricted by algebraic constraints imposed by conservation laws, balance relations, circuit laws, or geometric restrictions; see, e.g., \cite{blajer1992index, winkler2003stochastic, romisch2003stochastic, wang2020manifold}. Such systems are naturally described by stochastic differential-algebraic equations (SDAEs), where stochastic differential dynamics and algebraic constraints are coupled in a single formulation \cite{winkler2003stochastic, romisch2003stochastic, chen2025strong, serea2025existence, tsafack2025pathwise, tsafack2025strong, tsafack2026pathwise}. Over the past few decades, the constant singular-matrix setting has served as a standard framework in much of the existing theory for SDAEs \cite{winkler2003stochastic, alabert2006linear}. In this case, the differential and algebraic subspaces are fixed, and the associated projectors can be chosen independently of time. Beyond this setting, the singular matrix itself may vary with time, reflecting time-varying parameters, changing physical coefficients, or evolving constraint structures. Such a dependence may also affect the differential and algebraic subspaces, leading to time-dependent projectors and additional terms in the reduced dynamics; see, e.g., \cite{chen2025strong, serea2025existence, tsafack2025pathwise, tsafack2025strong, tsafack2026pathwise}. As a result, the analysis becomes structurally more delicate than in the constant singular-matrix setting.

In the present paper, we focus on an intermediate but nontrivial case in which the singular matrix is allowed to depend on time, whereas the differential-algebraic splitting remains fixed. This setting extends the constant singular-matrix case, but still preserves a fixed projector structure, which is crucial for the structure-preserving weak convergence analysis developed below. More precisely, we consider the following  SDAEs
\begin{equation}\label{eq:sdae}
      A_{t}\,dX_{t}
      =
      F(t,X_{t})\,dt + G(t,X_{t})\,dW_{t},
      \quad t \in (0,T]
\end{equation}
with initial value $X_{0} \in \R^{d}$. Here $\{W_{t}\}_{t \in [0,T]}$ is an $m$-dimensional standard Brownian motion defined on a complete filtered probability space $(\Omega, \F, \P, \{\F_{t}\}_{t \in [0,T]})$, where the filtration $\{\F_{t}\}_{t \in [0,T]}$ satisfies the usual conditions. The coefficients $F \colon [0,T] \times \R^{d} \to \R^{d}$ and $G \colon [0,T] \times \R^{d} \to \R^{d \times m}$ are given functions, and the matrix $A_{t} \in \R^{d \times d}$ is singular for every $t \in [0,T]$. The singularity of $A_{t}$ is responsible for the appearance of algebraic constraints and distinguishes \eqref{eq:sdae} from standard unconstrained SDEs. Throughout this paper, \eqref{eq:sdae} is assumed to be an index-$1$ SDAE in the sense that the diffusion is compatible with the algebraic constraints, and the algebraic variables are globally uniquely determined by the differential variables; see, e.g., \cite{winkler2003stochastic, serea2025existence}. SDAEs of the form \eqref{eq:sdae} provide a useful framework for modelling constrained stochastic evolution phenomena in which differential dynamics and algebraic restrictions coexist. Typical examples include constrained dynamics in protein folding problems, constrained Brownian motion in nanostructured materials, and transient noise simulation of microelectronic circuits; see, e.g., \cite{neumaier1997molecular, winkler2003stochastic, raccis2011confined}.

Since explicit solutions of SDAEs are rarely available, numerical approximations are indispensable for studying their dynamical behaviour. However, unlike standard unconstrained SDEs, numerical methods for SDAEs must not only approximate the differential dynamics accurately, but also preserve the algebraic constraints at each time step; see, e.g., \cite{oki2023improved}. This constraint-preserving requirement constitutes one of the main structural issues in the numerical analysis of SDAEs. For SDAEs with a constant singular matrix, a number of numerical methods and convergence results have been developed under global Lipschitz or related assumptions \cite{kupper2012rungekutta, qin2019general, sickenberger2009local, winkler2003stochastic}. By contrast, the numerical analysis of SDAEs with time-dependent singular matrices is much less developed, and existing results mainly concern pathwise or strong convergence of numerical approximations \cite{chen2025strong, tsafack2025pathwise, tsafack2025strong, tsafack2026pathwise}. In many applications, however, statistical observables such as mean values and response functionals play a central role, especially in Monte Carlo simulation and uncertainty quantification. Their accurate approximation requires weak convergence analysis, which focuses on errors in expectations of solution functionals; see, e.g., \cite{chen2026weak, milstein2004stochastic, wang2024weak, zhao2025weak}. As far as we know, a systematic weak convergence analysis of structure-preserving methods for SDAEs with time-dependent singular matrices has not yet been established. This motivates us to investigate the weak convergence order of structure-preserving stochastic theta methods for SDAEs \eqref{eq:sdae}.

Given a uniform stepsize $h = \frac{T}{N}, N \in \N$ and the grid $t_{n} := nh, n = 0, 1, \cdots, N$, the stochastic theta method for each $\theta \in (0,1]$ applied to \eqref{eq:sdae} is given by
\begin{equation}\label{eq:theta-scheme}
\left\{\begin{aligned}
      &A_{t_n}Y_{n+1}
      =
      A_{t_n}Y_n
      +
      h\big( (1-\theta)F(t_n,Y_n) + \theta F(t_{n+1},Y_{n+1}) \big)
      +
      G(t_n,Y_n)\Delta{W_n},
      \\
      &Y_{0} = X_{0}
\end{aligned}\right.
\end{equation}
for $n = 0,1,\cdots,N-1$, where $Y_{n}$ denotes the numerical approximation of $X_{t_n}$ and $\Delta{W_n} := W_{t_{n+1}} - W_{t_n}$. It is worth emphasizing that the explicit case $\theta = 0$ is excluded from the present analysis, since the explicit update neither automatically enforces the algebraic constraint at the new time level nor uniquely determines the algebraic component of $Y_{n+1}$ in general; see Remark \ref{rk:choiceoftheta} for more details. Thus the implicit contribution corresponding to $\theta > 0$ is essential for both well-posedness and constraint preservation. A central ingredient of the present work is to exploit the index-$1$ algebraic-differential decomposition at both the continuous and discrete levels, following the framework developed in \cite{chen2025strong}. At the continuous level, the exact solution admits the decomposition $X_{t} = U_{t} + \widehat{V}(t,U_{t})$, where $U_{t} = PX_{t}$ is the differential component and $\widehat{V}$ is the algebraic reconstruction map, so that the original SDAE can be represented through a reduced SDE for $U_{t}$. At the discrete level, we do not discretize the reduced SDE directly; instead, we analyze the stochastic theta method as a scheme formulated at the original SDAE level and extract the induced one-step dynamics of its differential component.  This viewpoint provides the bridge between the constraint-preserving property of the original SDAE-level scheme and the Kolmogorov-function-based weak error analysis for the induced differential one-step approximation.

Since $A_{t_n}$ is singular, the well-posedness of the implicit update and the preservation of the algebraic constraints are not automatic. We address both issues  by exploiting the fixed projector structure and the algebraic reconstruction map: the SDAE-level update is reduced to a fixed-point problem for the differential component, while the algebraic component is recovered by the algebraic reconstruction map so that the new constraint is enforced. This proves  that for sufficiently small stepsizes and for every $\theta \in (0,1]$, the stochastic theta method admits a unique adapted numerical solution and preserves the constraint manifold at each time level; see Theorem~\ref{thm:wp_invariance}. We then turn to the weak convergence order. Since the stochastic theta method is formulated at the original SDAE level, its weak error cannot be analyzed by directly applying the classical SDE theory to the reduced equation. Instead, we introduce the backward Kolmogorov function associated with the reduced SDE and establish an abstract weak convergence theorem for constraint-preserving one-step approximations. This theorem reduces the global weak convergence problem to a one-step weak local error estimate for the induced differential component. For the stochastic theta method \eqref{eq:theta-scheme}, verifying this local weak error estimate is the main technical step. We insert a frozen Euler reference step between the exact reduced flow and the induced differential one-step update. In this way, the local weak error is decomposed into two parts: the weak expansion error of the frozen Euler reference step and the weak contribution arising from the implicit correction. The former is handled by a standard Brownian increment expansion, while the latter is controlled by using the regularity of the Kolmogorov function and suitable increment estimates. Combining this local estimate with the moment bounds of the numerical solution yields weak convergence of order one, as stated in Theorem  \ref{thm:weak_order_one_theta}, namely, for sufficiently smooth test functions $\varphi \colon \R^{d} \to \R$ with polynomial growth,
\begin{align*}
      \big|\E\big[\varphi(X_{T})\big]
      - \E\big[\varphi(Y_{N})\big]\big|
      \leq
      Ch.
\end{align*}

The main contributions of this paper are threefold. First, we establish an abstract weak convergence theorem for constraint-preserving one-step approximations of index-$1$ SDAEs. Second, we prove that the stochastic theta method is well posed and preserves the algebraic constraint manifold for sufficiently small stepsizes and $\theta \in (0,1]$. Third, by combining the  Kolmogorov-function approach with a frozen Euler reference step and an implicit correction estimate, we verify the required local weak error estimate and obtain weak convergence of order one for the stochastic theta method. The remainder of this paper is organized as follows. Section \ref{sec:basicsetting} introduces the index-$1$ SDAE setting, the algebraic reconstruction map, and the reduced SDE for the differential component. In Section \ref{sec:mainbody}, we develop the weak convergence framework and verify its assumptions for the stochastic theta method. Finally, Section \ref{sec:experiments} presents some numerical experiments illustrating both the constraint-preserving property and the weak convergence order.

\section{Index-$1$ SDAE setting and reduced SDE}\label{sec:basicsetting}
\label{sec:setting}

In this section, we introduce the structural assumptions on the index-$1$ SDAEs \eqref{eq:sdae} and derive the associated reduced dynamics for the differential component. To this end, we first fix some notation. The letter $C$ denotes a generic positive constant whose value may change from line to line, but is independent of the stepsize of the numerical method. Let $\langle \cdot,\cdot \rangle$ and $|\cdot|$ denote the Euclidean inner product and norm in $\R^{d}$, respectively. For any matrix $B \in \R^{d \times m}$, $B^{\top}$ denotes its transpose, $|B| := \sqrt{\operatorname{Tr}(B^{\top}B)}$ its Frobenius norm, $\operatorname{Im}(B) := \{Bx : x \in \R^{m}\}$ its image, and $\operatorname{Ker}(B) := \{x \in \R^m : Bx = 0\}$ its kernel. For $k \in \{1,2,\cdots\}$, let $C^{k}(\R^{d};\R)$ denote the class of all real-valued functions on $\R^{d}$ whose partial derivatives up to order $k$ exist and are continuous. We further denote by $C_{\mathrm{pol}}^{k}(\R^{d};\R)$ the class of all functions $\psi \in C^{k}(\R^{d};\R)$ such that for every multi-index $\alpha$ with $|\alpha| \leq k$, the derivative $D^{\alpha}\psi$ has at most polynomial growth.

\subsection{Matrix structure, constraint manifold, and exact solution}
We first impose a structural assumption on $A_{t}$, which allows the singular matrix to vary in time while keeping the differential-algebraic splitting and the associated projectors independent of time.

\begin{assumption}\label{ass:At-structure}
      For any $t \in [0,T]$, the singular value decomposition of matrix $A_t$ takes the form $A_t = M\Sigma_t N$, where $M,N \in \R^{d \times d}$ are orthogonal matrices and $\Sigma_t = \operatorname{diag}\bigl(\sigma_1(t), \cdots, \sigma_r(t), 0, \cdots, 0\bigr)$ is continuously differentiable with $r\in\{1,2,\cdots,d-1\}$. Moreover, there exist constants $\underline{\sigma}$ and $\overline{\sigma}$ such that
      \begin{align*}
            0 < \underline{\sigma} 
            \leq \sigma_i(t) 
            \leq \overline{\sigma} 
            < +\infty,
            \quad i = 1,2,\cdots,r, t \in [0,T].
      \end{align*}
\end{assumption}

Since the Moore--Penrose inverse of any real matrix always exists and is unique, we denote $A_t^{-}$ the unique Moore--Penrose inverse matrix of $A_t$ for each $t \in [0,T]$. It follows from Assumption \ref{ass:At-structure} that $A_t^{-} = N^{\top} \Sigma_t^{-} M^{\top}$ with $$\Sigma_t^{-} = \operatorname{diag}\bigl(\sigma_1^{-1}(t), \cdots, \sigma_r^{-1}(t), 0, \cdots, 0\bigr), \quad t \in [0,T].$$ Now we define the time independent operator
\begin{equation*}
      P := A_t^{-}A_t = N^{\top}\begin{pmatrix}
            I_{r} & 0 \\ 0 & 0_{d-r} \end{pmatrix}N,
\end{equation*}
which is a projector along $\mathrm{Ker}(A_t)$, where $0_{d-r}$ represents the $(d-r) \times (d-r)$ zero matrix. The projector $Q := I-P$ onto $\mathrm{Ker}(A_t)$ is such that $A_tQ = 0$ and the projector $R := I-A_tA_t^{-}$ along $\mathrm{Im}(A_t)$ satisfies $RA_t = 0$. We next impose the index-$1$ structural conditions. 

\begin{assumption}\label{ass:sdae-structure}
      Let the initial value $X_0$ be deterministic and satisfy $RF(0,X_0) = 0$.
      Besides, the algebraic Jacobian $J(t,x) := A_t + R D_xF(t,x)$ is invertible for each $(t,x) \in [0,T] \times \R^{d}$ and there exists a constant $L_{J} > 0$ such that $\sup_{(t,x) \in [0,T] \times \R^{d}} |J(t,x)^{-1}| \leq L_{J}$. Moreover, the diffusion coefficient is compatible with the algebraic constraint, namely,
            \begin{equation*} RG(t,x)=0, \quad (t,x)\in[0,T]\times\mathbb R^d. \end{equation*}
            
\end{assumption}

The following standard growth and Lipschitz conditions guarantee the well-posedness of the associated stochastic dynamics and provide the basic moment estimates needed later.
\begin{assumption}\label{ass:coeff}
      Let the coefficients $F \colon [0,T] \times \R^{d} \to \R^{d}$ and $G \colon [0,T] \times \R^{d} \to \R^{d \times m}$ satisfy the global Lipschitz and linear growth conditions with respect to the space variable, i.e., there exist constants $L, C > 0$ such that for all $t \in [0,T]$ and $x,y \in \R^{d}$,
      \begin{gather*}
            |F(t,x)-F(t,y)| + |G(t,x)-G(t,y)| \leq L|x-y|,
            \\
            |F(t,x)|+|G(t,x)| \leq C(1+|x|).
      \end{gather*}
\end{assumption}

Under the preceding assumptions, the SDAE \eqref{eq:sdae} admits a unique continuous adapted solution, which remains on the algebraic constraint manifold.
\begin{proposition}\label{prop:exact_constraint}
      Suppose that Assumptions \ref{ass:At-structure}, \ref{ass:sdae-structure} and \ref{ass:coeff} hold. Then  \eqref{eq:sdae} admits a unique continuous adapted solution $\{X_{t}\}_{t \in [0,T]}$, which satisfies
      \begin{equation}\label{eq:sdae_integral}
            \int_{0}^{t} A_s\,dX_{s}
            =
            \int_{0}^{t} F(s,X_{s})\,ds
            +
            \int_{0}^{t} G(s,X_{s})\,dW_s,
            \quad t \in [0,T].
      \end{equation}
      Moreover, the exact solution $\{X_{t}\}_{t \in [0,T]}$ satisfies the algebraic constraint
      \begin{equation*}
            X_{t} \in \mathcal{M}_t, \quad t \in [0,T], \quad \mathbb{P}\text{\rm{-a.s.}},
      \end{equation*}
      where $\mathcal{M}_{t} := \{x \in \R^{d} : RF(t,x) = 0\}$ for each $t \in [0,T]$.  
\end{proposition}

\begin{proof}
      The existence of unique solution $\{X_{t}\}_{t \in [0,T]}$ to the index-1 SDAE \eqref{eq:sdae} follows from \cite{serea2025existence}. Multiplying the integral equation \eqref{eq:sdae_integral} by $R$, we obtain
      \begin{equation*}
            R\int_{0}^{t} A_s\,dX_{s}
            =
            \int_{0}^{t} R F(s,X_{s})\,ds
            +
            \int_{0}^{t} R G(s,X_{s})\,dW_s, \quad t \in [0,T].
      \end{equation*}
      It follows from $R A_{s} = 0$ and $R G(s,X_{s}) = 0$ for all $s \in [0,T]$ that there exists a set $\Omega_{0} \in \mathcal{F}$ with $\P(\Omega_0) = 1$ such that for every $\omega \in \Omega_{0}$,
      \begin{equation*}
            \int_{0}^{t} RF(s,X_{s}(\omega))\,ds = 0, \quad t \in [0,T].
      \end{equation*}
      For each fixed $\omega \in \Omega_{0}$, the map
      \begin{equation*}
            t \mapsto \int_{0}^{t} RF(s,X_{s}(\omega))\,ds
      \end{equation*}
      is absolutely continuous and is identically zero on $[0,T]$. Therefore, by the fundamental theorem of calculus for Lebesgue integrals, we obtain $RF(t,X_t(\omega)) = 0$ for almost every $t \in [0,T]$. Since $\{X_{t}\}_{t \in [0,T]}$ has continuous sample paths and the coefficient $F(t,x)$ is continuous in $(t,x)$, the map $t \mapsto RF(t,X_t(\omega))$ is continuous on $[0,T]$. From the fact that a continuous function which is equal to zero almost everywhere on $[0,T]$ must be identically zero, we conclude that for every $\omega \in \Omega_{0}$,
      \begin{equation*}
            RF(t,X_t(\omega))=0, \quad t \in [0,T],
      \end{equation*}
      which implies $\P(X_t \in \mathcal M_t\ \text{for all } t\in[0,T]) = 1$ and thus completes the proof.
\end{proof}

\subsection{Algebraic reconstruction and the reduced SDE}
We now use the index-$1$ structure to eliminate the algebraic component and to obtain the  reduced SDE for the differential component. The key point is that the algebraic equation has a unique solution in the algebraic subspace for each differential variable. This gives rise to the algebraic reconstruction map. More precisely, for every $(t,u) \in [0,T] \times \R^{d}$, the algebraic equation
\begin{equation}\label{eq:Vhat_alg}
      A_t v + R F\bigl(t,u+v\bigr)=0
\end{equation}
admits a unique solution $v = \widehat{V}(t,u) \in \operatorname{Im}(Q)$; see, e.g., \cite{winkler2003stochastic, serea2025existence}. We call $\widehat{V}$ the algebraic reconstruction map. This map represents every constrained state in terms of its differential component and will now be applied to the exact solution of \eqref{eq:sdae}. For every $t \in [0,T]$, setting $U_{t} := PX_{t} \in \operatorname{Im}(P)$ and $V_{t} := QX_{t} \in \operatorname{Im}(Q)$ gives the decomposition $X_{t} = U_{t} + V_{t}$. Together with $RF(t,X_{t}) = 0$ and $V_{t} \in \operatorname{Im}(Q) = \ker(A_t)$ for all $t \in [0,T]$, we obtain
\begin{equation*}
      A_tV_t+RF(t,U_t+V_t)
      =
      A_tQX_t+RF(t,X_t)
      =
      0, \quad t \in [0,T],
\end{equation*}
which means that $V_{t}$ is a solution of the algebraic equation \eqref{eq:Vhat_alg} with $u = U_{t}$. By the uniqueness of the algebraic solution, we have $V_t = \widehat{V}(t,U_t), t \in [0,T]$, and consequently
\begin{equation}\label{eq:X_decomp}
      X_{t} = U_{t} + \widehat{V}(t,U_{t}), \quad t \in [0,T].
\end{equation}

To derive the evolution equation for the differential component, we multiply \eqref{eq:sdae} by $A_t^{-}$ and use $A_t^{-}A_t = P$ to deduce that
\begin{equation*}
      P\,dX_t = A_t^{-}F(t,X_t)\,dt + A_t^{-}G(t,X_t)\,dW_t. 
\end{equation*}
Together with $U_t = P X_t$ and \eqref{eq:X_decomp}, we derive the unconstrained SDE
\begin{equation}\label{eq:sde}
      U_t-U_0
      =
      \int_{0}^{t} A_s^{-}F\big(s,U_s+\widehat{V}(s,U_s)\big)\,ds
      +
      \int_{0}^{t} A_s^{-}G\big(s,U_s+\widehat{V}(s,U_s)\big)\,dW_s,
      \quad t \in [0,T].
\end{equation}
By introducing the reduced coefficients
\begin{equation}\label{eq:coff}
      f(t,u) := A_t^{-}F\big(t,u+\widehat{V}(t,u)\big),
      \quad
      g(t,u) := A_t^{-}G\big(t,u+\widehat{V}(t,u)\big),
      \quad t \in [0,T], u \in \R^{d},
\end{equation}
one can rewrite the reduced SDE \eqref{eq:sde} as follows
\begin{equation}\label{eq:reducedsde}
      U_t-U_0
      =
      \int_{0}^{t} f(s,U_s)\,ds
      +
      \int_{0}^{t} g(s,U_s)\,dW_s,
      \quad t \in [0,T].
\end{equation}
Before studying the well-posedness of \eqref{eq:reducedsde}, we record several results of the matrix structure.

\begin{lemma}\label{lem:A-inv-Lipschitz}
      Suppose that Assumption \ref{ass:At-structure} holds. Then there exists a constant 
      $C_A := \max_{t \in [0,T]} |\Sigma_t'| \in [0, +\infty)$ such that 
      \begin{equation*}
            |A_t| \leq \sqrt r\,\overline{\sigma},
            \quad
            |A_t^{-}| \leq \frac{\sqrt r}{\underline{\sigma}},
            \quad
            |A_t'| \leq C_A, \quad t \in [0,T].              
      \end{equation*}
      Besides, the map $t \mapsto A_t^{-}$ is differentiable, and there exists a constant $C := \frac{C_A}{\underline{\sigma}^{2}} \geq 0$ such that
      \begin{equation*}
            \bigg|\frac{d}{dt} A_t^{-}\bigg| \leq C,
            \quad |A_t^{-}-A_s^{-}| \leq C|t-s|,
            \quad s,t \in [0,T].
      \end{equation*}
\end{lemma}

\begin{proof}
      Since $M,N$ are orthogonal, the facts $A_t=M\Sigma_tN$ and $A_t^- = N^\top \Sigma_t^- M^\top$ give that
      \begin{equation*}
            |A_t|=|\Sigma_t|,
            \quad
            |A_t^{-}|=|\Sigma_t^{-}|,
            \quad
            |A_t'|=|\Sigma_t'|,
            \quad t \in [0,T].
      \end{equation*}
      It follows from $\Sigma_t = \operatorname{diag}\bigl(\sigma_1(t), \cdots, \sigma_r(t), 0, \cdots, 0\bigr)$ and 
      $\Sigma_t^{-} = \operatorname{diag}(\sigma_1(t)^{-1},\cdots,\sigma_r(t)^{-1},0,\cdots,0)$ with $r\in\{1,2,\cdots,d-1\}$ that
      \begin{equation*}
            |A_{t}|^{2} = \sum_{i=1}^{r} \sigma_{i}(t)^2 \leq r \overline{\sigma}^{2},
            \quad
            |A_t^{-}|^{2} = \sum_{i=1}^{r} (\sigma_{i}(t))^{-2} \leq \frac{r}{\underline{\sigma}^{2}},  
      \end{equation*}
      which yields $|A_t| \leq \sqrt r\,\overline{\sigma}$ and $|A_t^{-}| \leq \frac{\sqrt r}{\underline{\sigma}}$. Noting that the map $t \mapsto \Sigma_t'$ is continuous on $[0,T]$, one gets $|A_t'| \leq \max_{t \in [0,T]}|\Sigma_t'| =: C_{A} < +\infty$.

      Owing to $\sigma_i \in C^1([0,T];\R)$ and $\sigma_i(t) \geq \underline{\sigma}>0$ for $i = 1,2,\cdots,r, t \in [0,T]$, the map $t\mapsto \sigma_i(t)^{-1}$ belongs to $C^1([0,T];\R)$ with
      \begin{equation*}
            \frac{d}{dt}\sigma_i(t)^{-1}
            =
            -\sigma_i'(t)\sigma_i(t)^{-2},
            \quad i=1,\cdots,r,
      \end{equation*}
      which implies that $t\mapsto \Sigma_t^{-}$ is differentiable and satisfies
      \begin{equation*}
            \frac{d}{dt}\Sigma_t^{-}
            =
            \operatorname{diag}\bigl(
            -\sigma_1'(t)\sigma_1(t)^{-2}, \cdots, -\sigma_r'(t)\sigma_r(t)^{-2}, 0, \cdots, 0
            \bigr).
      \end{equation*}
      Applying $\frac{d}{dt} A_t^{-} = N^\top(\frac{d}{dt}\Sigma_t^{-})M^\top$ yields
      \begin{align*}
            \bigg|\frac{d}{dt} A_t^{-}\bigg|
            =
            \bigg|\frac{d}{dt}\Sigma_t^{-}\bigg| 
            =
            \left(
            \sum_{i=1}^r \frac{|\sigma_i'(t)|^2}{\sigma_i(t)^4}
            \right)^{1/2}
            \leq
            \frac{|\Sigma_t'|}{\underline{\sigma}^{2}}
            \leq
            \frac{C_A}{\underline{\sigma}^{2}}.
      \end{align*}
      Moreover, the fundamental theorem of calculus shows that for any $s,t \in [0,T]$ with $s \leq t$, 
      \begin{align*}
            |A_t^{-}-A_s^{-}|
            =
            \left|\int_s^t \frac{d}{dr} A_r^{-}\,dr\right|
            \leq
            \int_s^t \bigg|\frac{d}{dr} A_r^{-}\bigg|\,dr
            \leq
            C|t-s|.
      \end{align*}
      Thus we complete the proof.
\end{proof}

As a direct consequence of the index-$1$ invertibility condition, one can derive the basic Lipschitz and growth properties of the algebraic reconstruction map $\widehat{V}$; see \cite[Lemma 2.2]{chen2025strong} for its proof.

\begin{lemma}\label{lem:Vhat_basic}
      Suppose that Assumptions \ref{ass:At-structure}, \ref{ass:sdae-structure} and \ref{ass:coeff} hold. Then there exists a constant $L_V > 0$ such that
      \begin{equation}\label{eq:Vhat-Lip-u}
            \big|\widehat{V}(t,u_{1}) - \widehat{V}(t,u_{2})\big| \leq L_{V}\,|u_{1} - u_{2}|,
            \quad t \in [0,T], u_{1}, u_{2} \in \R^{d}.
      \end{equation}
      Moreover, there exists a constant $C > 0$ such that
      \begin{gather}
            \label{eq:Vhat-Lip-t}
            \bigl|\widehat{V}(t,u)-\widehat{V}(s,u)\bigr| \leq C(1+|u|)|t-s|,
            \quad s, t \in [0,T], u \in \R^{d},
            \\
            \label{eq:Vhat-linear-growth}
            |\widehat{V}(t,u)| \leq C(1+|u|),
            \quad t \in [0,T], u \in \R^{d}.
      \end{gather}
\end{lemma}

The following proposition verifies the global Lipschitz continuity and linear growth of the reduced coefficients, and consequently establishes the well-posedness of the reduced SDE \eqref{eq:reducedsde}.
\begin{proposition}\label{lem:_coeff}
      Suppose that Assumptions \ref{ass:At-structure}, \ref{ass:sdae-structure} and \ref{ass:coeff} hold. Then there exists a constant $C > 0$ such that for all $t \in [0,T]$ and $u, u_1, u_2 \in \R^{d}$,
      \begin{gather}
            \label{eq:f-g-lipschitz}
            |f(t,u_1)-f(t,u_2)|+|g(t,u_1)-g(t,u_2)| \leq C|u_1-u_2|,
            \\\label{eq:linear_growth_fg}
            |f(t,u)|+|g(t,u)| \leq C(1+|u|).
      \end{gather}
      Moreover, the reduced SDE \eqref{eq:reducedsde} admits a unique strong solution $\{U_{t}\}_{t \in [0,T]}$.
\end{proposition}

\begin{proof}
      For any $t \in [0,T]$ and $u \in \R^{d}$, setting $\Psi_t(u) := u + \widehat{V}(t,u)$ yields
      \begin{align*}
            f(t,u) = A_t^{-}F(t,\Psi_t(u)),
            \quad
            g(t,u) = A_t^{-}G(t,\Psi_t(u)).
      \end{align*}
      By the boundedness of $A_t^{-}$ in Lemma \ref{lem:A-inv-Lipschitz}, the global Lipschitz continuity of $F$, $G$ in Assumption \ref{ass:coeff}, and the Lipschitz continuity of $\widehat{V}(t,\cdot)$ in Lemma \ref{lem:Vhat_basic}, we have 
      \begin{align*}
            |f(t,u_1)-f(t,u_2)|
            \leq&~
            |A_t^{-}| \big|F(t,\Psi_t(u_1))-F(t,\Psi_t(u_2))\big|  
            \\\leq&~
            C|\Psi_t(u_1)-\Psi_t(u_2)|  
            \\\leq&~
            C\bigl(|u_1-u_2|
            +
            |\widehat{V}(t,u_1) - \widehat{V}(t,u_2)|\bigr)  
            \\\leq&~
            C|u_1-u_2|, \quad t \in [0,T], u_1, u_2 \in \R^{d}.
      \end{align*}
      The same arguments give $|g(t,u_1)-g(t,u_2)| \leq C|u_1-u_2|, t \in [0,T], u_1, u_2 \in \R^{d}$. Thus \eqref{eq:f-g-lipschitz} holds. Next, using the linear growth of $F$, $G$ in Assumption \ref{ass:coeff}, the boundedness of $A_t^{-}$ in Lemma \ref{lem:A-inv-Lipschitz}, and the linear growth of $\widehat{V}$ in Lemma \ref{lem:Vhat_basic}, we obtain that for any $t \in [0,T]$ and $u \in \R^{d}$,
      \begin{align*}
            |f(t,u)| + |g(t,u)|
            \leq
            |A_t^{-}|\bigl(|F(t,\Psi_t(u))| + |G(t,\Psi_t(u))|\bigr)  
            \leq
            C(1 + |\Psi_t(u)|)  
            \leq
            C(1+|u|),
      \end{align*}
      which proves \eqref{eq:linear_growth_fg}. According to \cite{mao2008stochastic}, \eqref{eq:reducedsde} admits a unique strong solution $\{U_{t}\}_{t \in [0,T]}$. 
\end{proof}

\subsection{Regularity of the reconstruction map and the reduced flow}
To establish the regularity of the backward Kolmogorov function used in the weak convergence analysis, we impose higher-order smoothness assumptions on the original coefficients and transfer them to the algebraic reconstruction map, the reduced coefficients, and the reduced stochastic flow. In view of Assumption \ref{ass:coeff}, the first-order spatial derivative bounds follow from global Lipschitz continuity once differentiability is available, so the following assumption starts from the second-order derivatives.

\begin{assumption}
\label{ass:FG-in-H}
      Fix $p \in \N$. For every $t \in [0,T]$, assume that $F(t,\cdot) \in C^{2p+2}(\R^{d};\R^{d})$ and $G(t,\cdot) \in C^{2p+2}(\R^{d};\R^{d \times m})$. Moreover, there exists a constant $C > 0$ such that
      \begin{equation*}
            |D_{x}^{j}F(t,x)| + |D_{x}^{j}G(t,x)| \leq C, 
            \quad t \in [0,T], x \in \R^{d}, j = 2, 3, \cdots, 2p+2.   
      \end{equation*}
\end{assumption}

We first derive the high-order spatial regularity of the algebraic reconstruction map from the implicit function theorem and the uniform invertibility of the algebraic Jacobian.
\begin{lemma}\label{lem:Vhat-spatial-regularity}
      Suppose that Assumptions \ref{ass:At-structure}, \ref{ass:sdae-structure}, \ref{ass:coeff} and \ref{ass:FG-in-H} hold. Then for every $t \in [0,T]$, the map $u \mapsto \widehat{V}(t,u)$ belongs to $C^{2p+2}(\mathbb R^d;\operatorname{Im}(Q))$. Moreover, for each $j = 1, 2, \cdots, 2p+2$, there exists a constant $C > 0$, independent of $t \in [0,T]$ and $u \in \R^{d}$, such that
      \begin{align*}
            \big|D_u^j\widehat{V}(t,u)\big| \leq C,
            \quad t \in [0,T],u \in \R^{d}.
      \end{align*}
      Here $D_u^j\widehat{V}(t,u)$ is understood as a $j$-linear map, and $|\cdot|$ denotes the corresponding operator norm.
\end{lemma}

\begin{proof}
      Fix $t \in [0,T]$ and define $H_t(u,v) := A_{t}v + RF(t,u+v)$ for $u, v \in \R^{d}$. Since $F(t,\cdot) \in C^{2p+2}(\R^{d};\R^{d})$, it follows that $H_t \in C^{2p+2}(\R^{d} \times \R^{d};\R^{d})$. Moreover, the derivative of $H_t$ with respect to the algebraic variable $v$ is
      $$D_vH_t(u,v) = A_t + RD_xF(t,u+v) = J(t,u+v).$$
      Let $u_0 \in \R^{d}$ be arbitrary and set $v_0 := \widehat{V}(t,u_0)$. Then $H_t(u_0,v_0) = 0$. Since $D_{v}H_{t}(u_{0},v_{0}) = J(t,u_0 + v_0)$ is invertible by Assumption \ref{ass:sdae-structure}, the implicit function theorem implies that there exist a neighbourhood $U_0$ of $u_0$ and a unique function $\widetilde{V} \in C^{2p+2}(U_0;\R^{d})$ such that $H_t(u,\widetilde{V}(u)) = 0$ for any $u \in U_{0}$. By the global uniqueness of the algebraic equation defining $\widehat{V}(t,u)$, this local solution coincides with $\widehat{V}(t,u)$ on $U_0$. Since $u_0$ is arbitrary, we obtain $\widehat{V}(t,\cdot) \in C^{2p+2}(\R^{d};\operatorname{Im}(Q))$.

      It remains to prove the uniform boundedness of the derivatives. For fixed $t$, letting $V(u) := \widehat{V}(t,u)$ and $\Psi(u) := u + V(u)$ enables us to rewrite
      \begin{equation}\label{eq:AtVuRFtPsiu}
            A_tV(u) + RF(t,\Psi(u)) = 0.
      \end{equation}
      First, differentiating \eqref{eq:AtVuRFtPsiu} with respect to $u$ gives $A_tDV(u) + RD_xF(t,\Psi(u))D\Psi(u) = 0$. Since $D\Psi(u) = I+DV(u)$, we have
      $$J(t,\Psi(u))DV(u)
        =
        \big(A_t+RD_xF(t,\Psi(u))\big)DV(u)
        =
        -RD_xF(t,\Psi(u)).$$
      Using $RD_xF(t,\Psi(u)) = J(t,\Psi(u))-A_t$ yields $DV(u) = J(t,\Psi(u))^{-1}A_t-I$ and accordingly
      \begin{align}\label{eq:boundedDVu}
            |D_u \widehat{V}(t,u)| = |DV(u)| \leq |J(t,\Psi(u))^{-1}|\,|A_t|+|I| \leq C,  
      \end{align}
      where we used the uniform boundedness of $J^{-1}$ and $A_t$.

      We now prove the higher-order estimates by induction. Let $2 \leq j \leq 2p+2$, and let $\xi_1,\cdots,\xi_j \in \R^{d}$. Applying $D_u^j$ to \eqref{eq:AtVuRFtPsiu} and using the multivariate Fa\`a di Bruno formula, we get
      \begin{equation}\label{eq:AtDjVuequalzero}
            A_tD^jV(u)[\xi_1,\cdots,\xi_j] 
            +
            R D^j\big(F(t,\Psi(u))\big)[\xi_1,\cdots,\xi_j] = 0
      \end{equation}
      with 
      \begin{align*}
            D^j\big(F(t,\Psi(u))\big)[\xi_1,\cdots,\xi_j]
            =
            \sum_{\pi\in\Pi_j} D_x^{|\pi|}F(t,\Psi(u))
            \Big[D^{|B_1|}\Psi(u)[\xi_{B_1}], \cdots,
            D^{|B_{|\pi|}|}\Psi(u)[\xi_{B_{|\pi|}}]\Big],
      \end{align*}
      where $\Pi_j$ denotes the set of all partitions $\pi = \{B_1,\cdots,B_{|\pi|}\}$ of the set $\{1, \cdots,j\}$. For a block $B = \{i_1,\cdots,i_\ell\}$, we use the notation $D^{|B|}\Psi(u)[\xi_B] := D^\ell\Psi(u)[\xi_{i_1},\cdots,\xi_{i_\ell}]$. The term corresponding to the one-block partition $\pi = \big\{\{1,\cdots,j\}\big\}$ is
      $$D_xF(t,\Psi(u))D^j\Psi(u)[\xi_1,\cdots,\xi_j].$$
      Letting $\pi_0 := \big\{\{1,\cdots,j\}\big\}$ yields $|\pi_0| = 1$ and $B_1 = \{1,\cdots,j\}$ as well as
      \begin{align*}
            D_x^{|\pi_0|}F(t,\Psi(u))\Big[D^{|B_1|}\Psi(u)[\xi_{B_1}]\Big]
            =
            D_xF(t,\Psi(u))\Big[D^j\Psi(u)[\xi_1,\cdots,\xi_j]\Big].          
      \end{align*}
      It follows that
      \begin{align*}
            D^j\big(F(t,\Psi(u))\big)[\xi_1,\ldots,\xi_j] 
            =
            D_xF(t,\Psi(u))\Big[D^j\Psi(u)[\xi_1,\ldots,\xi_j]\Big]  
            +
            \Gamma_j(t,u)[\xi_1,\ldots,\xi_j],
      \end{align*}
      where
      \begin{align*}
            \Gamma_j(t,u)[\xi_1,\ldots,\xi_j]
            :=
            \sum_{\substack{\pi\in\Pi_j\\ |\pi|\ge2}}
            D_x^{|\pi|}F(t,\Psi(u))\Big[D^{|B_1|}\Psi(u)[\xi_{B_1}],
            \cdots, D^{|B_{|\pi|}|}\Psi(u)[\xi_{B_{|\pi|}}]\Big].
      \end{align*}
      Applying $D^j\Psi(u) = D^jV(u), j \geq 2$ and \eqref{eq:AtDjVuequalzero} leads to
      \begin{align*}
            J(t,\Psi(u))D^jV(u)[\xi_1,\ldots,\xi_j]
            =      
            \big(A_t+RD_xF(t,\Psi(u))\big)D^jV(u)[\xi_1,\ldots,\xi_j]
            =
            -R\Gamma_j(t,u)[\xi_1,\ldots,\xi_j].          
      \end{align*}
      We now observe that $\Gamma_j(t,u)$ contains only lower-order derivatives of $\Psi$. Indeed, if $|\pi| \geq 2$, then every block $B_\ell$ of the partition satisfies $|B_\ell| \leq j-1$. Hence $\Gamma_j(t,u)$ only involves $D^\ell\Psi(u), 1 \leq \ell \leq j-1$. Moreover, it holds that
      \begin{align}\label{eq:DlpuequalDVu}
            D\Psi(u) = I + DV(u),
            \quad
            D^\ell\Psi(u) = D^\ell V(u),\quad \ell \geq 2.          
      \end{align}

      As we have already proved the boundedness of $DV(u)$ in \eqref{eq:boundedDVu}, now we assume, as the
      induction hypothesis, that for some $j \geq 2$,
      $$|D^\ell V(u)| \leq C, \quad \ell=1,\cdots,j-1,$$
      uniformly in $t \in [0,T]$ and $u \in \R^{d}$. It follows from \eqref{eq:DlpuequalDVu} that $|D^\ell\Psi(u)| \leq C$ for $\ell = 1,\cdots,j-1$. Since $j \leq 2p+2$ and the number of partitions of $\{1,\cdots,j\}$ is finite, we further use Assumption \ref{ass:FG-in-H} to obtain $|\Gamma_j(t,u)| \leq C$ and consequently
      \begin{align*}
            |D^jV(u)|
            \leq
            |J(t,\Psi(u))^{-1}|\,|R|\,|\Gamma_j(t,u)|
            \leq C.
      \end{align*}
      This proves the induction step and completes the proof. 
\end{proof}

We further impose a time regularity assumption on the original coefficients, which is needed to handle the time-inhomogeneity of the reduced SDE and the local weak expansion.
\begin{assumption}\label{ass:time_reg_merged}
      The coefficients $F, G$ are continuously differentiable with respect to the time variable. Moreover, there exist constants $C > 0$ and $l \geq 1$ such that 
      \begin{gather*}
            \label{eq:FG-partialtime}
            |\partial_t F(t,x)|+|\partial_t G(t,x)|
            \leq C\big(1+|x|^{l}\big),
            \quad t \in [0,T], x \in \R^{d}.
      \end{gather*}
\end{assumption}

The next lemma transfers the regularity of $F$, $G$, and $\widehat{V}$ to the reduced coefficients $f$ and $g$. In particular, it provides the $C^{1,2}$-regularity needed for the backward Kolmogorov equation. For later use, we write $f = (f^1,\cdots,f^d)^{\top} \in \R^{d}$ and $g = (g^{i,j})_{1 \leq i \leq d,\;1 \leq j \leq m}$.

\begin{lemma}
\label{lem:_coeff}
      Suppose that Assumptions \ref{ass:At-structure}, \ref{ass:sdae-structure}, \ref{ass:coeff}, \ref{ass:FG-in-H} and \ref{ass:time_reg_merged} hold. Then for every $t \in [0,T]$, the coefficients $f, g$ in \eqref{eq:reducedsde} satisfy $f(t,\cdot) \in C^{2p+2}(\R^{d};\R^{d})$ and $g(t,\cdot) \in C^{2p+2}(\R^{d};\R^{d \times m})$. Moreover, there exists a constant $C > 0$ such that for all $t \in [0,T]$ and $u \in \R^{d}$,
      \begin{gather}
            \label{eq:f-g-high-der-growth}
            |D_u^j f(t,u)| + |D_u^j g(t,u)| \leq C, \quad j = 1,\cdots,2p+2,
            \\\label{eq:f-g-time-derivative-growth}
            |\partial_t f(t,u)| + |\partial_t g(t,u)| \leq C(1+|u|^l).
      \end{gather}
      Consequently, $f^{i} \in C^{1,2}([0,T] \times \R^{d};\R)$ and $g^{i,j} 
      \in C^{1,2}([0,T] \times \R^{d};\R)$ for $i = 1,\cdots,d, j = 1,\cdots,m$.
\end{lemma}

\begin{proof}
      For any $t \in [0,T]$ and $u \in \R^{d}$, setting $\Psi_t(u) := u + \widehat{V}(t,u)$ yields $f(t,u) = A_t^{-}F(t,\Psi_t(u))$ and $g(t,u) = A_t^{-}G(t,\Psi_t(u))$. 
      Since $F(t,\cdot) \in C^{2p+2}(\R^{d};\R^{d})$, $G(t,\cdot) \in C^{2p+2}(\R^{d};\R^{d \times m})$ and $\widehat{V}(t,\cdot) \in C^{2p+2}(\R^{d};\operatorname{Im}(Q))$ by Lemma \ref{lem:Vhat-spatial-regularity}, we have $u \mapsto \Psi_t(\cdot)$ belonging to $C^{2p+2}(\R^{d};\R^{d})$. Utilizing the standard regularity of compositions ensures $F(t,\Psi_t(\cdot)) \in C^{2p+2}(\R^{d};\R^{d})$, $G(t,\Psi_t(\cdot)) \in C^{2p+2}(\R^{d};\R^{d \times m})$ and accordingly $f(t,\cdot) \in C^{2p+2}(\R^{d};\R^{d})$, $g(t,\cdot) \in C^{2p+2}(\R^{d};\R^{d \times m})$. As $A_t^{-}$ does not depend on $u$, we use the multivariate Fa\`a di Bruno formula to show that for any $1 \leq j \leq 2p+2$ and $\xi_1, \cdots, \xi_j \in \R^{d}$, 
      \begin{align*}
            D_u^j f(t,u)[\xi_1,\cdots,\xi_j]
            =&~
            A_t^{-}D_u^j\bigl(F(t,\Psi_t(u))\bigr)[\xi_1,\cdots,\xi_j]
            \\=&~
            A_t^{-} \sum_{\pi\in\Pi_j} D_x^{|\pi|}F(t,\Psi_t(u))
            \Big[D_u^{|B_1|}\Psi_t(u)[\xi_{B_1}], \cdots, 
            D_u^{|B_{|\pi|}|}\Psi_t(u)[\xi_{B_{|\pi|}}]\Big],
      \end{align*}
      where $\Pi_j$ denotes the set of all partitions $\pi = \{B_1,\cdots,B_{|\pi|}\}$ of $\{1,\cdots,j\}$. For a block $B = \{i_1,\cdots,i_\ell\}$, we write $D_u^{|B|}\Psi_t(u)[\xi_B] := D_u^\ell\Psi_t(u)[\xi_{i_1}, \cdots,\xi_{i_\ell}]$. For $j = 1$, the boundedness of $D_xF$ follows from the global Lipschitz continuity of $F(t,\cdot)$ together with its differentiability. For $j \geq 2$, the boundedness of $D_x^kF$, $2 \leq k \leq j$ follows from Assumption \ref{ass:FG-in-H}. Moreover, using $D_u\Psi_t(u) = I + D_u\widehat{V}(t,u)$ and $D_u^\ell\Psi_t(u)=D_u^\ell\widehat{V}(t,u), \ell \geq 2$ as well as Lemma \ref{lem:Vhat-spatial-regularity} implies 
      \begin{align*}
            |D_u^\ell\Psi_t(u)| \leq C, \quad \ell = 1,\cdots,2p+2,         
      \end{align*}
      uniformly for $t \in [0,T]$ and $u \in \R^{d}$. It follows from the boundedness of $A_t^{-}$ in Lemma \ref{lem:A-inv-Lipschitz} that
      \begin{align*}
            |D_u^j f(t,u)|
            \leq
            |A_t^{-}| \sum_{\pi \in \Pi_j} 
            \Big|D_x^{|\pi|}F(t,\Psi_t(u))\Big|
            \prod_{B \in \pi}
            \Big| D_u^{|B|}\Psi_t(u)\Big|
            \leq C,
            \quad j = 1,\cdots,2p+2.
      \end{align*}
      The same arguments applied to $G(t,\Psi_t(u))$ yield $|D_u^j g(t,u)| \leq C$ for $j = 1,\cdots,2p+2$, i.e., \eqref{eq:f-g-high-der-growth} holds.

      To prove \eqref{eq:f-g-time-derivative-growth}, we first estimate $\partial_t \widehat{V}(t,u)$. From $A_t\widehat{V}(t,u) + RF(t,u + \widehat{V}(t,u)) = 0$ and the parameter-dependent implicit function theorem, the map $t \mapsto \widehat{V}(t,u)$ is continuously differentiable for each fixed $u$. Differentiating $A_t\widehat{V}(t,u) + RF(t,u + \widehat{V}(t,u)) = 0$ with respect to \(t\), we get
      \begin{align*}
            J(t,\Psi_t(u))\partial_t \widehat{V}(t,u)
            =
            -A_t' \widehat{V}(t,u)
            -
            R\partial_tF(t,\Psi_t(u)).
      \end{align*}
      Using the boundedness of $J^{-1}$ in Assumption \ref{ass:sdae-structure}, the boundedness of $A_t'$ in Lemma \ref{lem:A-inv-Lipschitz}, the linear growth of $\widehat{V}$ in Lemma \ref{lem:Vhat_basic}, and the time regularity assumption on \(F\) in Assumption \ref{ass:time_reg_merged} yields
      \begin{align*}
            |\partial_t \widehat{V}(t,u)|
            \leq&~
            |J(t,\Psi_t(u))^{-1}|\big(| A_t'|\,|\widehat{V}(t,u)|
            + |R|\,|\partial_t F(t,\Psi_t(u))|\bigr)  
            \\\leq&~
            C\big(1 + |u| + 1 + |\Psi_t(u)|^l\big)  
            \\\leq&~
            C(1 + |u|^l)
      \end{align*}
      due to $l \geq 1$ and $|\Psi_t(u)| \leq C(1+|u|)$. Now differentiating $f$ and $g$ with respect to $t$ gives
     \begin{gather*}
          \partial_t f(t,u)
          =
          \bigg(\frac{d}{dt} A_t^{-}\bigg)F(t,\Psi_t(u))
          +
          A_t^{-}\Big(\partial_t F(t,\Psi_t(u))
          +
          D_x F(t,\Psi_t(u))\partial_t \widehat{V}(t,u)\Big), \\
          \partial_t g(t,u)
          =
          \bigg(\frac{d}{dt} A_t^{-}\bigg)G(t,\Psi_t(u))
          +
          A_t^{-}\Big(\partial_t G(t,\Psi_t(u))
          +
          D_x G(t,\Psi_t(u))\partial_t \widehat{V}(t,u)\Big).
        \end{gather*}
      By the boundedness of $A_t^{-}$ and $\frac{d}{dt} A_t^{-}$, the linear growth of $F$ and $G$, the time regularity of $F$ and $G$, the boundedness of $D_xF$ and $D_xG$, and the estimate for $\partial_t \widehat{V}$, we get
      \begin{align*}
            |\partial_t f(t,u)| + |\partial_t g(t,u)|
            \leq&~
            \bigg|\frac{d}{dt} A_t^{-}\bigg|\big(|F(t,\Psi_t(u))| 
            + 
            |G(t,\Psi_t(u))|\big)  
            \\&~+
            |A_t^{-}|\big(|\partial_tF(t,\Psi_t(u))|
            + |\partial_tG(t,\Psi_t(u))|\big)  
            \\&~+
            |A_t^{-}|\big(|D_xF(t,\Psi_t(u))| + |D_xG(t,\Psi_t(u))|\big)
            |\partial_t \widehat{V}(t,u)|  
            \\\leq&~
            C(1 + |\Psi_t(u)|) + C(1 + |\Psi_t(u)|^l) + C(1 + |u|^l)  
            \\\leq&~
            C(1 + |u|^l),
      \end{align*}
      which implies \eqref{eq:f-g-time-derivative-growth}.

      Finally, we verify the $C^{1,2}$-regularity. It suffices to show that $\partial_t f, \partial_t g, D_u f, D_u g, D_u^2 f, D_u^2 g$ are jointly continuous in $(t,u) \in [0,T] \times \R^{d}$. Indeed, from \eqref{eq:coff} and $\Psi_t(u) = u + \widehat{V}(t,u)$, we have
      \begin{gather*}
            \partial_t f(t,u)
            =
            \bigg(\frac{d}{dt} A_t^{-}\bigg)F(t,\Psi_t(u))
            +
            A_t^{-}\big(\partial_tF(t,\Psi_t(u))
            + 
            D_xF(t,\Psi_t(u))\partial_t\widehat{V}(t,u)\big),
            \\
            \partial_t g(t,u)
            =
            \bigg(\frac{d}{dt} A_t^{-}\bigg)G(t,\Psi_t(u))
            +
            A_t^{-}\big(\partial_tG(t,\Psi_t(u))
            +
            D_xG(t,\Psi_t(u))\partial_t\widehat{V}(t,u)\big),
            \\
            D_u f(t,u)
            =
            A_t^{-}D_xF(t,\Psi_t(u))D_u\Psi_t(u),
            \quad
            D_u g(t,u)
            =
            A_t^{-}D_xG(t,\Psi_t(u))D_u\Psi_t(u),
            \\
            D_u^2 f(t,u)[\xi,\eta]
            =
            A_t^{-}\big(D_x^2F(t,\Psi_t(u))
            [D_u\Psi_t(u)\xi,D_u\Psi_t(u)\eta] 
            +
            D_xF(t,\Psi_t(u))D_u^2\Psi_t(u)[\xi,\eta]\big),
            \\
            D_u^2 g(t,u)[\xi,\eta]
            =
            A_t^{-}\big(D_x^2G(t,\Psi_t(u))
            [D_u\Psi_t(u)\xi,D_u\Psi_t(u)\eta]
            +
            D_xG(t,\Psi_t(u))D_u^2\Psi_t(u)[\xi,\eta]\big)
      \end{gather*}
      for any $\xi,\eta \in \R^{d}$. By the assumptions on $F$ and $G$, the regularity of $A_t^{-}$, and the parameter-dependent implicit-function formulas for $\widehat{V}$, the quantities $\widehat{V}, \partial_t \widehat{V}, D_u \widehat{V}, D_u^2 \widehat{V}$ are jointly continuous in $(t,u)$. It follows that $\Psi_t(u), \partial_t \widehat{V}(t,u), D_u \Psi_t(u), D_u^2 \Psi_t(u)$ are also jointly continuous, and that $\partial_t f, \partial_t g, D_u f, D_u g, D_u^2 f, D_u^2 g$ are jointly continuous on $[0,T] \times \R^{d}$. The proof is complete. 
\end{proof}

With the regularity of $f$ and $g$ established, standard stochastic-flow results can be applied to the reduced SDE. For $0 \leq t \leq s \leq T$ and $u \in \R^{d}$, denote by $U(t,u;s)$ the solution value at time $s$ of \eqref{eq:reducedsde} with initial condition $U(t,u;t) = u$.  

\begin{proposition}\label{lem:flow_regularity_reduced}
      Suppose that Assumptions \ref{ass:At-structure}, \ref{ass:sdae-structure}, \ref{ass:coeff}, \ref{ass:FG-in-H} and \ref{ass:time_reg_merged} hold. Then for every $q \geq 2$, there exists a constant $C > 0$, independent of $t, s, u$, such that
      \begin{equation}\label{eq:moment_exact_solution_all_q}
            \sup_{0 \leq t \leq s \leq T}
            \E\left[\sup_{t \leq r \leq s}|U(t,u;r)|^{q}\right]
            \leq
            C(1+|u|^q), \quad u \in \R^{d}.
      \end{equation} 
      Moreover, the reduced SDE \eqref{eq:reducedsde} admits a version such that for every $0 \leq t \leq s \leq T$, the map $u \mapsto U(t,u;s)$ is $2p+2$ times continuously differentiable. Furthermore, for every $j = 1, 2, \cdots, 2p+2$ and $q \geq 2$, there exists a constant $C > 0$, independent of $t,s,u$, such that
      \begin{equation}\label{eq:ass9-derivative-weighted}
            \sup_{0 \leq t \leq s \leq T}\sup_{u \in \R^{d}}
            \E\left[\sup_{t \leq r \leq s}|D_u^jU(t,u;r)|^{q}\right]
            \leq 
            C.
      \end{equation}
      Consequently, $u \mapsto U(t,u;T)$ is $2p+2$ times mean-square differentiable with respect to the initial value $u$ in the componentwise sense. Finally, the flow property
      \begin{equation*}
            U(t,u;r)
            =
            U\bigl(s,U(t,u;s);r\bigr),
            \quad
            0 \leq t \leq s \leq r \leq T
      \end{equation*}
      holds almost surely.
\end{proposition}

\begin{proof}
      The standard moment estimate for SDEs with globally Lipschitz and linearly growing coefficients yields \eqref{eq:moment_exact_solution_all_q}; see, e.g., \cite[Theorem 2.4.4]{mao2008stochastic}.

      By Lemma \ref{lem:_coeff}, the classical differentiability theorem for stochastic flows thus applies; see, e.g., \cite[Theorem 2.4.3]{kunita1984stochastic}. It follows that the solution admits a version such that $u \mapsto U(t,u;s)$ is $2p+2$ times continuously differentiable. The derivative estimates follow from the corresponding variational equations. Indeed, the first derivative satisfies a linear SDE with bounded coefficients, while the $j$-th derivative satisfies a linear SDE whose nonhomogeneous terms involve only derivatives of the flow of order less than $j$. Using the Burkholder--Davis--Gundy inequality, the Gronwall inequality, and induction over $j$, one obtains \eqref{eq:ass9-derivative-weighted}.

      Finally, the flow property follows from the pathwise uniqueness of the strong solution. Indeed, for $0 \leq t \leq s \leq r \leq T$, both $U(t,u;r)$ and $U(s,U(t,u;s);r)$ solve the same SDE on $[s,T]$ with the same initial value $U(t,u;s)$ at time $s$. Hence they are indistinguishable. The proof is complete.
\end{proof}

\section{Weak convergence analysis of the stochastic theta method}\label{sec:mainbody}
In this section, we prove the weak convergence order of the stochastic theta method. We first establish an abstract weak convergence theorem for constraint-preserving one-step approximations of the index-$1$ SDAEs \eqref{eq:sdae}. The theorem reduces the global weak error estimate to a one-step weak local error estimate for the induced differential component, together with suitable moment bounds of the numerical solution. We then verify these conditions for the stochastic theta method.

\subsection{An abstract weak convergence theorem}
Let $0 = t_0<t_1<\cdots<t_N = T$ be a uniform partition of $[0,T]$ with stepsize $h = \frac{T}{N}, N \in \N$, and set $\Delta{W_n} := W_{t_{n+1}}-W_{t_n}$, $n=0,1,\cdots,N-1$. For any $t \in [0,T], x \in \R^{d}$ and $0 < t+h \leq T$, we introduce a one-step approximation in the form
\begin{equation}\label{eq:one_step_map}
      Y(t,x;t+h)
      =
      x + \Phi(t,x,h;\xi),
\end{equation}
where $\xi$ is a random variable, possibly vector-valued, and $\Phi$ is a measurable function valuing in $\R^{d}$. The corresponding discrete numerical solution $\{Y_n\}_{n=0}^{N}$ is then defined recursively by
\begin{equation}\label{eq:numerical_recursion}
      Y_0 = X_0, \quad
      Y_{n+1} = Y(t_n,Y_n;t_{n+1}) = Y_n + \Phi(t_n,Y_n,h;\xi_n),
      \quad n = 0,1,\cdots,N-1,
\end{equation}
where $\xi_{n}$ is independent of $Y_0,Y_1,\cdots,Y_n,\xi_0,\xi_1,\cdots,\xi_{n-1}$. The following theorem gives the abstract weak convergence principle used in this paper. It reduces the global weak error estimate to constraint
preservation, a local weak error estimate for the induced differential component, and moment bounds for the numerical solution.

\begin{theorem}\label{thm:weak_convergence_sdae}
      Let $\{Y_n\}_{n = 0}^{N}$ be defined by \eqref{eq:numerical_recursion}, and set $u_n := PY_n$ for $n = 0, 1, \cdots, N$. Suppose that Assumptions \ref{ass:At-structure}, \ref{ass:sdae-structure}, \ref{ass:coeff}, \ref{ass:FG-in-H} and \ref{ass:time_reg_merged} hold. Fix $\varphi \in C_{\mathrm{pol}}^{2p+2}(\R^{d};\R)$, define $$\psi(u) := \varphi\bigl(u + \widehat{V}(T,u)\bigr),
      \quad w(t,u) := \E\big[\psi(U(t,u;T))\big], 
      \quad t \in [0,T], u \in \R^{d},$$
      and assume that the following conditions are satisfied.
      \begin{enumerate}
            \item[\textup{(i)}] The numerical approximation $\{Y_n\}_{n = 0}^{N}$ is well defined and preserves the constraint manifold along the numerical trajectory, i.e., $\P(Y_{n} \in \mathcal{M}_{t_{n}}) = 1$ for all $n = 0,1,\cdots,N$.

            \item[\textup{(ii)}] The numerical approximation $\{Y_n\}_{n = 0}^{N}$ satisfies the moment bound, i.e., for any $q \geq 1$, there exists a constant $C > 0$ such that
            \begin{align*}
                  \sup_{N \in \N}\sup_{0 \leq n \leq N} \E\big[|Y_{n}|^{q}\big]
                  \leq
                  C(1 + |Y_{0}|^{q}).
            \end{align*}
            
            \item[\textup{(iii)}] The induced differential one-step map $$\widetilde{u}(t,u;t+h) := PY\bigl(t,u + \widehat{V}(t,u);t+h\bigr), \quad t \in [0,T-h], u \in \R^{d}$$ 
            satisfies $u_{n+1} = \widetilde{u}(t_n,u_n;t_{n+1})$, whenever $Y_n = u_n + \widehat{V}(t_n,u_n)$. 
            
            \item[\textup{(iv)}] There exist constants $C > 0$ and $q \geq 0$ such that for every $h \in (0,T]$ and $u \in \R^{d}$,
            \begin{align*}
                  \sup_{t\in[0,T-h]}
                  \left|\E\bigl[w(t+h,U(t,u;t+h))\bigr]
                  - \E\bigl[w(t+h,\widetilde{u}(t,u;t+h))\bigr]\right|
                  \leq
                  C(1+|u|^{q})h^{p+1}.                
            \end{align*}
      \end{enumerate}
      Then there exist constants $C > 0$ and $q \geq 0$, independent of $h$ and $N$, such that
      \begin{align*}
            \big|\E\big[\varphi(X_T)\big] - \E\big[\varphi(Y_N)\big]\big|
            \leq
            C(1+|Y_{0}|^{q})h^{p}.
      \end{align*}
\end{theorem}

\begin{proof}
      From \textup{(i)} and the global solvability of the algebraic equation \eqref{eq:Vhat_alg}, it follows that $Y_n = u_n + \widehat{V}(t_n,u_n), n = 0,1,\cdots,N$. Together with $X_T = U(T) + \widehat{V}(T,U(T))$ and $\psi(u) := \varphi(u + \widehat{V}(T,u))$, we obtain
      \begin{align}\label{eq:transform_to_w_correct_B}
            \E\big[\varphi(X_T)\big] - \E\big[\varphi(Y_N)\big]
            =&~
            \E\big[\varphi\big( U(T) + \widehat{V}(T,U(T)) \big)\big]
            -
            \E\big[\varphi\big(u_N + \widehat{V}(T,u_N)\big)\big] \notag
            \\=&~
            \E\big[\psi(U(T))\big] - \E\big[\psi(u_N)\big] \notag 
            \\=&~
            \E\big[\psi(U(t_0,u_0;T))\big]
            -
            \E\big[\psi(U(t_N,u_N;T))\big] \notag
            \\=&~
            \E\big[w(t_0,u_0)\big]
            -
            \E\big[w(t_N,u_N)\big].
      \end{align}
      By the flow property of the strong solution, it holds that for any $0 \leq t \leq s \leq T$ and $u \in \R^{d}$,
      \begin{align*}
            U(t,u;T) = U(s,U(t,u;s);T), \quad \P\text{-a.s.}        
      \end{align*}
      Since $f$ and $g$ are deterministic and the Brownian increments after time $s$ are independent of $\F_s$, the solution $U(t,u;\cdot)$ is a time-inhomogeneous Markov process \cite{mao2008stochastic}, we have
      \begin{equation}\label{eq:w_semigroup_B}
            w(t,u)
            =
            \E\big[\psi(U(s,U(t,u;s);T))\big]  
            =
            \E\big[\E\big(\psi(U(s,U(t,u;s);T))
            \mid \F_s \big) \big]  
            =
            \E\big[w(s,U(t,u;s))\big]
      \end{equation}
      for any $0 \leq t \leq s \leq T$, which implies $\E\big[w(t_{n},u_{n})\big] = \E\big[w\big(t_{n+1},U(t_{n},u_{n};t_{n+1})\big)\big], n = 0,1,\cdots,N-1$.
      As a consequence, we obtain
      \begin{align}\label{eq:A4_sde_B}
            \E\big[w(t_0,u_0)\big] - \E\big[w(t_N,u_N)\big]
            =
            \sum_{i=0}^{N-1}\big(
            \E\big[w(t_{i+1},U(t_i,u_i;t_{i+1}))\big]
            -
            \E\big[w(t_{i+1},u_{i+1})\big]\big).
      \end{align}
      Using $u_{n+1} = \widetilde{u}(t_{n},u_{n};t_{n+1})$, \textup{(iv)} and $|u_{n}| \leq |P|\,|Y_{n}|$ enables us to get
      \begin{align*}
            &~\big|\E\big[w(t_0,u_0)\big] - \E\big[w(t_N,u_N)\big]\big|
            \\\leq&~
            \sum_{n=0}^{N-1}\big|
            \E\big[w(t_{n+1},U(t_{n},u_{n};t_{n+1}))\big]
            -
            \E\big[w(t_{n+1},\widetilde{u}(t_{n},u_{n};t_{n+1}))\big]\big|
            \\\leq&~
            Ch^{p+1} \sum_{n=0}^{N-1} \big(1 + \E\big[|u_{n}|^{q}\big]\big)
            \leq
            Ch^{p+1} \sum_{n=0}^{N-1} \big(1 + \E\big[|Y_{n}|^{q}\big]\big),
      \end{align*}
      which together with  \textup{(ii)} implies the desired result and thus completes the proof.
\end{proof}

\subsection{Verification for the stochastic theta method}
We now verify the assumptions of Theorem \ref{thm:weak_convergence_sdae} for the stochastic theta method. The verification consists of well-posedness and constraint preservation, moment estimates, regularity of the Kolmogorov function, and the one-step weak local error estimate.

\subsubsection{Well-posedness and constraint preservation}
We first prove that the stochastic theta method is well defined and preserves the algebraic constraints at all time levels.
\begin{proposition}\label{thm:wp_invariance}
      Suppose that Assumptions \ref{ass:At-structure}, \ref{ass:sdae-structure}, \ref{ass:coeff}, \ref{ass:FG-in-H} and \ref{ass:time_reg_merged} hold, and let $\theta \in (0,1]$. Then there exists a constant $h_{\mathrm{wp}} \in (0,1)$, independent of $n$, such that for every $h \in (0,h_{\mathrm{wp}})$, the stochastic theta method \eqref{eq:theta-scheme} admits a unique adapted numerical solution $\{Y_n\}_{n = 0}^{N}$ and satisfies $\P(Y_{n} \in \mathcal{M}_{t_n}) = 1$ for every $n = 0,1,\cdots,N$.
\end{proposition}

\begin{proof}
      Since $Y_{0} = X_{0} \in \mathcal{M}_{0}$, it suffices to argue by mathematical induction. Assume that $\P(Y_{n} \in \mathcal{M}_{t_n}) = 1$ for some fixed $n$. Letting $u_n := PY_n$ gives $Y_n = u_n + \widehat{V}(t_n,u_n)$. Since $P = A_t^{-}A_t$, we have $PA_t^{-} = A_t^{-}$, and hence $\operatorname{Im}(A_t^{-})\subseteq \operatorname{Im}(P)$. Define
      \begin{align*}
            Z_n
            :=
            u_n + h(1-\theta)A_{t_n}^{-}F(t_n,Y_n)
            +
            A_{t_n}^{-}G(t_n,Y_n)\Delta W_n \in \operatorname{Im}(P).
      \end{align*}
      and $\Gamma_n(u)
            :=
            Z_n
            +
            h\theta A_{t_n}^{-}
            F\bigl(t_{n+1},u+\widehat{V}(t_{n+1},u)\bigr),
            u \in \operatorname{Im}(P)$.
      Noting that \(\widehat{V}(t_{n+1},u) \in \operatorname{Im}(Q)\) and $A_{t_n}^{-}F\bigl(t_{n+1},u+\widehat{V}(t_{n+1},u)\bigr) \in \operatorname{Im}(A_{t_n}^{-}) \subseteq\operatorname{Im}(P)$, we have $\Gamma_{n} \colon \operatorname{Im}(P) \to \operatorname{Im}(P)$.
      For any $u_1, u_2 \in \operatorname{Im}(P)$, using the Lipschitz continuity of $F$ with respect to the space variable in Assumption \ref{ass:coeff} and the Lipschitz continuity of $\widehat{V}(t,\cdot)$ in Lemma \ref{lem:Vhat_basic}, we obtain
      \begin{align*}
            |\Gamma_n(u_1)-\Gamma_n(u_2)|
            \leq&~
            h\theta |A_{t_n}^{-}|\,
            \bigl| F(t_{n+1},u_1 + \widehat{V}(t_{n+1},u_1))
            - F(t_{n+1},u_2+\widehat{V}(t_{n+1},u_2)) \bigr|  
            \\\leq&~
            h\theta |A_{t_n}^{-}|L\big(|u_1-u_2|
            + |\widehat{V}(t_{n+1},u_1) - \widehat{V}(t_{n+1},u_2)|\big) 
            \\\leq&~
            hK_{\theta} |u_1-u_2|,
      \end{align*}
      where $K_{\theta} := \theta \sup_{t\in[0,T]}|A_t^{-}|\,L(1+L_V)$.  Choosing $h_{\mathrm{wp}} \in (0,1)$ such that $h_{\mathrm{wp}}K_{\theta} < 1$, $\Gamma_{n}$ is a contraction on the complete metric space $\operatorname{Im}(P)$ for any $h \in (0,h_{\mathrm{wp}})$. It follows that there exists a unique fixed point $u_{n+1} \in \operatorname{Im}(P)$ satisfying
      \begin{equation}\label{eq:un-fixed-point}
            u_{n+1}
            =
            Z_n
            +
            h\theta A_{t_n}^{-}F\big(t_{n+1},
            u_{n+1}+\widehat{V}(t_{n+1},u_{n+1})\big).
      \end{equation}
      Defining
      \begin{equation}\label{eq:Y-def_final}
            Y_{n+1} := u_{n+1} + \widehat{V}(t_{n+1},u_{n+1})
      \end{equation}
      and using the property of $\widehat{V}$, we have $RF(t_{n+1},Y_{n+1}) = 0$, which implies $\P(Y_{n+1} \in \mathcal{M}_{t_{n+1}}) = 1$.

      We next show that $Y_{n+1}$ defined by \eqref{eq:Y-def_final} indeed solves \eqref{eq:theta-scheme}. Letting
      \begin{align*}
            \mathcal{E}_{n}
            :=
            A_{t_n}Y_{n+1} - A_{t_n}Y_n
            - h\big( (1-\theta)F(t_n,Y_n) + \theta F(t_{n+1},Y_{n+1})\big)  
            - G(t_n,Y_n)\Delta{W_n}
      \end{align*}
      and using $A_{t_n}^{-}A_{t_n} = P$, $PY_n = u_n$, $PY_{n+1} = u_{n+1}$ yield
      \begin{align*}
            A_{t_n}^{-}\mathcal{E}_{n}
            =
            u_{n+1} - u_n - h(1-\theta)A_{t_n}^{-}F(t_n,Y_n)
            -
            h\theta A_{t_n}^{-}F(t_{n+1},Y_{n+1})  
            -
            A_{t_n}^{-}G(t_n,Y_n)\Delta W_n.
      \end{align*}
      Owing to \eqref{eq:un-fixed-point}, we obtain $A_{t_n}^{-} \mathcal{E}_{n} = 0$, and thus $\mathcal{E}_{n} \in \ker(A_{t_n}^{-})$. For the Moore--Penrose inverse associated with the decomposition in Assumption \ref{ass:At-structure}, one has $\ker(A_{t_n}^{-}) = \operatorname{Im}(R)$. It follows that $\mathcal{E}_{n} \in \operatorname{Im}(R)$. Besides, applying $R$ to $\mathcal{E}_{n}$, and using $RA_{t_n} = 0, RG(t_n,\cdot)=0$ lead to
      \begin{align*}
            R\mathcal{E}_{n}
            =
            -h\big( (1-\theta)RF(t_n,Y_n) + \theta RF(t_{n+1},Y_{n+1}) \big).
      \end{align*}
      As $Y_{n} \in \mathcal{M}_{t_n}$  and $Y_{n+1}\in\mathcal M_{t_{n+1}}$, we have $RF(t_n,Y_n) = 0, RF(t_{n+1},Y_{n+1}) = 0$, which implies $R\mathcal{E}_{n} = 0$. Together with $\mathcal{E}_{n} \in \operatorname{Im}(R)$, we deduce that $\mathcal{E}_{n} = R\mathcal{E}_{n} = 0$, which means that $Y_{n+1}$ solves \eqref{eq:theta-scheme}. This proves existence.

      It remains to prove uniqueness. Let $\widetilde{Y}_{n+1}$ be any solution of \eqref{eq:theta-scheme}. Applying $R$ to \eqref{eq:theta-scheme} gives
      \begin{align*}
            0
            =
            h\big( (1-\theta)RF(t_n,Y_n)
            +
            \theta RF(t_{n+1},\widetilde Y_{n+1}) \big).
      \end{align*}
      Owing to $RF(t_n,Y_n) = 0$ and $\theta > 0$, we get $RF(t_{n+1},\widetilde{Y}_{n+1}) = 0$, 
      and therefore $\widetilde{Y}_{n+1} \in \mathcal{M}_{t_{n+1}}$. Setting $\widetilde{u}_{n+1} := P\widetilde{Y}_{n+1}$ and using the algebraic reconstruction gives $\widetilde{Y}_{n+1} = \widetilde{u}_{n+1}
      + \widehat{V}(t_{n+1},\widetilde{u}_{n+1})$. Applying $A_{t_n}^{-}$ to \eqref{eq:theta-scheme} gives $\widetilde{u}_{n+1} = Z_n + h\theta A_{t_n}^{-}F(t_{n+1},\widetilde{Y}_{n+1})$. It follows that
      \begin{align*}
            \widetilde{u}_{n+1}
            =
            Z_n
            +
            h\theta A_{t_n}^{-} F\big(t_{n+1},\widetilde{u}_{n+1}
            +
            \widehat{V}(t_{n+1},\widetilde{u}_{n+1})\big)
            =
            \Gamma_n(\widetilde{u}_{n+1}),
      \end{align*}
      that is to say, $\widetilde{u}_{n+1}$ is a fixed point of $\Gamma_n$. Since $\Gamma_n$ has a unique fixed point in $\operatorname{Im}(P)$, we have $\widetilde{u}_{n+1} = u_{n+1}$, and consequently $\widetilde{Y}_{n+1} = Y_{n+1}$. This proves uniqueness.
\end{proof}

The role of the restriction $\theta > 0$ is clarified in the following remark.
\begin{remark}\label{rk:choiceoftheta}
      The restriction $\theta > 0$ is essential for the present formulation. If $\theta = 0$, then the scheme reduces to
      \begin{equation*}
            A_{t_n}Y_{n+1}
            =
            A_{t_n}Y_n + hF(t_n,Y_n)
            +
            G(t_n,Y_n)\Delta{W_n}.
      \end{equation*}
      Since $A_{t_n}$ is singular, components of $Y_{n+1}$ in $\ker(A_{t_n}) = \operatorname{Im}(Q)$ are not determined by this equation. Hence the explicit update does not uniquely fix the algebraic component of $Y_{n+1}$. Moreover, applying $R$ to the explicit scheme only yields $RF(t_n,Y_n) = 0$, which gives no information about $RF(t_{n+1},Y_{n+1})$. Thus the constraint $Y_{n+1} \in \mathcal{M}_{t_{n+1}}$ is not automatically enforced.

      By contrast, when $\theta > 0$, applying $R$ to the stochastic $\theta$-scheme and using $RA_{t_n} = 0$, $RG(t_n,\cdot) = 0$, and $Y_n \in \mathcal{M}_{t_n}$, gives $h \theta RF(t_{n+1},Y_{n+1}) = 0$. It follows that $RF(t_{n+1},Y_{n+1}) = 0$, and hence $Y_{n+1} \in \mathcal{M}_{t_{n+1}}$. The next iterate can then be uniquely reconstructed in the form
      \begin{equation*}
            Y_{n+1}
            =
            u_{n+1}
            +
            \widehat{V}(t_{n+1},u_{n+1}).
      \end{equation*}
      This explains why the well-posedness result is formulated for $\theta \in (0,1]$.
\end{remark}

\subsubsection{Moment estimates}
We next establish moment and increment estimates for the stochastic theta method. These estimates are needed to control the accumulation of local weak errors in the abstract weak convergence theorem. Since the method preserves the algebraic constraints, the numerical solution can be represented through its differential component and the algebraic reconstruction map. For convenience, set $u_n := PY_n, n = 0,1,\cdots,N$.

\begin{proposition}\label{lem:moment_Y}
      Suppose that Assumptions \ref{ass:At-structure}, \ref{ass:sdae-structure}, \ref{ass:coeff}, \ref{ass:FG-in-H} and \ref{ass:time_reg_merged} hold, and let $\theta \in (0,1]$. Then there exists a constant $h_{0} \in (0,h_{\mathrm{wp}}]$ such that for any $h \in (0,h_{0}]$ and any $q \geq 2$, there exists constants $C_{q},C > 0$, independent of $h,n,N$, such that
      \begin{gather}
            \label{eq:unqYnqmoment}
            \E\big[|u_{n}|^{q}\big] \leq C_{q}(1 + |u_{0}|^{q}),
            \quad
            \E\big[|Y_n|^{q}\big] \leq C_{q}(1 + |Y_{0}|^{q}),
            \quad n = 0,1,\cdots,N,
            \\\label{eq:Yn1minusYnq}
            \E\big(|Y_{n+1}-Y_n|^{q} \mid \F_{t_n} \big)
            \leq
            C_{q} h^{\frac{q}{2}}(1 + |Y_{n}|^{q}),
            \quad n = 0,1,\cdots,N-1,
            \\\label{eq:Yn1minusYn}
            \big|\E\big(Y_{n+1}-Y_{n} \mid \F_{t_n} \big)\big|
            \leq
            Ch(1 + |Y_{n}|^{2}),
            \quad n = 0,1,\cdots,N-1.
      \end{gather}
\end{proposition}

\begin{proof}
      For any $n = 0,1,\cdots,N$, using $u_n := PY_n$ and $\P(Y_n \in \mathcal{M}_{t_n}) = 1$ yields $Y_{n} = u_{n} + \widehat{V}(t_n,u_n)$. Together with $PA_t^{-} = A_t^{-}$ and \eqref{eq:coff}, applying $A_{t_n}^{-}$ to \eqref{eq:theta-scheme}  gives 
      \begin{align}\label{eq:Deltaun}
            u_{n+1}
            =
            u_n + h(1-\theta)f(t_n,u_n)
            +
            h\theta A_{t_n}^{-}F(t_{n+1},Y_{n+1})
            +
            g(t_n,u_n)\Delta{W_n}.          
      \end{align}
      We first derive several local estimates for $\Delta u_n :=u_{n+1}-u_n$. By the linear growth of $f, F, g$, the boundedness of $A_t^{-}$, and the linear growth of $\widehat{V}$ in Lemma \ref{lem:Vhat_basic}, there exists a constant $C_{0} > 0$, independent of $h,n,N$, such that for any $n = 0,1,\cdots,N$,
      \begin{align*}
            |\Delta u_n|
            \leq&~
            Ch(1+|u_n|) + Ch(1+|Y_{n+1}|)
            +
            C(1+|u_n|)|\Delta W_n|  
            \\\leq&~
            C_{0}h(1+|u_n|) + C_{0}h\bigl(1+|u_{n+1}|\bigr)
            +
            C_{0}(1+|u_n|)|\Delta W_n|  
            \\\leq&~
            C_{0}h(1+2|u_n|+|\Delta u_n|)
            +
            C_{0}(1+|u_n|)|\Delta W_n|.
      \end{align*}
      Choosing $h_{0} := \min\big\{h_{\mathrm{wp}},\frac{1}{2C_{0}},1\big\} > 0$ and employing $1-C_{0}h \geq \frac{1}{2}$ for any $h \in (0, h_0]$,  we obtain
      \begin{equation}\label{eq:Delta-u-pathwise-optimized}
            |\Delta u_n| \leq C(1+|u_n|)\bigl(h+|\Delta W_n|\bigr).
      \end{equation}
      It follows from $\E\big[|\Delta W_n|^{r}\big] \leq C_r h^{\frac{r}{2}}, r \geq 2$ that
      \begin{equation}\label{eq:Delta-u-local-moment-optimized}
            \E\big(|\Delta u_n|^{r} \mid \F_{t_n}\big)
            \leq
            C_{r} h^{\frac{r}{2}}(1+|u_n|^r).
      \end{equation}
      We also need a conditional first-moment estimate for $\Delta u_n$. By means of $Y_{n+1} = u_{n+1} + \widehat{V}(t_{n+1},u_{n+1})$, the linear growth of $\widehat{V}$ in Lemma \ref{lem:Vhat_basic} and $u_{n+1} = u_{n} + \Delta{u_{n}}$, we have
      \begin{align*}
            |Y_{n+1}|^2
            \leq
            C(1 + |u_{n+1}|^2)
            \leq
            C(1 + |u_n|^2 + |\Delta u_n|^2).
      \end{align*}
      Combining with \eqref{eq:Delta-u-pathwise-optimized} leads to
      \begin{align}\label{eq:EYnplus1Ftn}
           \E\big(|Y_{n+1}|^{2} \mid \F_{t_n}\big)
           \leq&~
           C(1+|u_n|^2)
           +
           C\E\big(|\Delta u_n|^2 \mid \F_{t_n}\big) \notag
           \\\leq&~
           C(1+|u_n|^2)
           +
           C(1+|u_n|^2)\E\bigl[(h+|\Delta W_n|)^2\bigr] \notag
           \\\leq&~
           C(1+|u_n|^2)+Ch(1+|u_n|^2) \notag
           \\\leq&~
           C(1+|u_n|^2).
      \end{align}
      Owing to \eqref{eq:Deltaun} and $\E\big(\Delta W_n \mid \F_{t_n}) = 0$, we have
      \begin{align*}
            \E\big(\Delta u_n \mid \F_{t_n}\big)
            =
            h(1-\theta)f(t_n,u_n)
            +
            h\theta A_{t_n}^{-}
            \E\big(F(t_{n+1},Y_{n+1}) \mid \F_{t_n} \big).
      \end{align*}
      Using the linear growth of $f$ and $F$, the boundedness of $A_t^{-}$ and \eqref{eq:EYnplus1Ftn}, we obtain
      \begin{equation}\label{eq:Delta-u-mean-optimized}
            \left|\E\big(\Delta u_n \mid \F_{t_n}\big)\right|
            \leq
            Ch(1+|u_n|).
      \end{equation}

      We now prove \eqref{eq:unqYnqmoment}. Applying the standard inequality
      \begin{align*}
            |x+y|^q
            \leq
            |x|^q + q|x|^{q-2} \langle x,y \rangle
            +
            C_q |x|^{q-2}|y|^2 + C_q |y|^q,
            \quad  q \geq 2, x,y \in \R^{d}
      \end{align*}
      with $x = u_n$ and $y = \Delta u_n$ (see, e.g., \cite{chen2026weak}) and taking conditional expectation show that
      \begin{align*}
            \E\big(|u_{n+1}|^q \mid \F_{t_n} \big)
            \leq&~
            |u_n|^q + C_q |u_n|^{q-1} \left|\E\big(\Delta u_n \mid \F_{t_n} \big)\right|  
            \\&~
            + C_q |u_n|^{q-2}\E\big( |\Delta u_n|^2 \mid \F_{t_n} \big)
            + C_q\E\big(|\Delta u_n|^{q} \mid \F_{t_n}\big).
      \end{align*}
      Because of \eqref{eq:Delta-u-local-moment-optimized} and \eqref{eq:Delta-u-mean-optimized}, we get
      \begin{align*}
            \E\big(|u_{n+1}|^q \mid \F_{t_n} \big)
            \leq&~
            |u_n|^q + C_qh|u_n|^{q-1}(1+|u_n|)
            +
            C_qh|u_n|^{q-2}(1+|u_n|^2)  
            + 
            C_qh^{q/2}(1+|u_n|^q)  
            \\\leq&~
            |u_n|^q+C_qh(1+|u_n|^q),
      \end{align*}
      and thus $\E\big[|u_{n+1}|^{q}\big] \leq (1+C_qh) \E\big[|u_n|^q\big] + C_{q}h$. Then the discrete Gronwall inequality implies the first result in \eqref{eq:unqYnqmoment}. Next, the linear growth of $\widehat{V}$ and $ Y_n = u_n + \widehat{V}(t_n,u_n)$ indicate $|Y_n|^q \leq C_q(1+|u_n|^q)$ and thus
      \begin{align*}      
            \E\big[|Y_n|^{q}\big]
            \leq
            C_q\left(1+ \E\big[|u_n|^q\big]\right)
            \le
            C_q(1+|u_0|^q),
      \end{align*}
      which together with $|u_0| \le |P|\,|Y_0|$ proves the second result in \eqref{eq:unqYnqmoment}.

      We now prove the increment estimate \eqref{eq:Yn1minusYnq}. Applying Lemma \ref{lem:Vhat_basic} yields
      \begin{align*}
            |Y_{n+1}-Y_n|
            \leq&~
            |\Delta u_n|
            +
            |\widehat{V}(t_{n+1},u_{n+1})-\widehat{V}(t_{n+1},u_n)|  
            +
            |\widehat{V}(t_{n+1},u_n)-\widehat{V}(t_n,u_n)|  
            \\\leq&~
            C|\Delta u_n| + Ch(1+|u_n|).
      \end{align*}
      It follow from \eqref{eq:Delta-u-local-moment-optimized} that
      \begin{align*}
            \E\big(|Y_{n+1}-Y_n|^q \mid \F_{t_n}\big)
            \leq
            C_q\E\big(|\Delta u_n|^q \mid \F_{t_n}\big)
            +
            C_qh^q(1+|u_n|^q)  
            \leq
            C_qh^{\frac{q}{2}}(1+|u_n|^q),
      \end{align*}
      which along with $|u_n| \leq |P|\,|Y_n|$ ensures \eqref{eq:Yn1minusYnq}. It remains to prove \eqref{eq:Yn1minusYn}. From the decomposition
      \begin{align*}
            Y_{n+1}-Y_n
            =&~
            \Delta u_n + \widehat{V}(t_{n+1},u_n) - \widehat{V}(t_n,u_n)  
            +
            \widehat{V}(t_{n+1},u_{n+1}) - \widehat{V}(t_{n+1},u_n)
            \\=&~
            \Delta u_n + \widehat{V}(t_{n+1},u_n) - \widehat{V}(t_n,u_n)  
            +
            D_u\widehat{V}(t_{n+1},u_n)\Delta u_n  
            \\&~
            + 
            \int_0^1(1-\rho) D_u^2 \widehat{V}(t_{n+1},u_n+\rho\Delta u_n)
            [\Delta u_n,\Delta u_n]\,d\rho,
      \end{align*}
      we utilize the bonunedness of $D_u\widehat{V}$ and $D_u^2\widehat{V}$ in Lemma \ref{lem:Vhat-spatial-regularity}, and the $\F_{t_n}$-measurability of $D_u\widehat{V}(t_{n+1},u_n)$ to derive that
      \begin{align*}
            \left|\E\big(Y_{n+1}-Y_n \mid \F_{t_n}\big)\right|
            \leq
            C\big|\E\big(\Delta u_n \mid \F_{t_n}\big)\big|
            +
            \big|\widehat{V}(t_{n+1},u_n) - \widehat{V}(t_n,u_n)\big|
            +
            C\E\big(|\Delta u_n|^2 \mid \F_{t_n}\big).
      \end{align*}
      Making use of \eqref{eq:Delta-u-local-moment-optimized}, \eqref{eq:Delta-u-mean-optimized} and the time-continuity estimate for $\widehat{V}$ in Lemma \ref{lem:Vhat_basic} results in
      \begin{align*}
            \left|\mathbb E\bigl[Y_{n+1}-Y_n\mid\mathcal F_{t_n}\bigr]\right|
            \leq
            Ch(1+|u_n|) + Ch(1+|u_n|) + Ch(1+|u_n|^2)  
            \leq
            Ch(1+|u_n|^2).
      \end{align*}
      which in combination with $|u_n| \leq |P|\,|Y_n|$ yields \eqref{eq:Yn1minusYn}. Thus, we complete the proof.
\end{proof}

\subsubsection{Regularity of the Kolmogorov function}
The forthcoming weak error analysis is based on the backward Kolmogorov function associated with the reduced SDE. Since only weak order one is needed for the stochastic theta method, we take $p = 1$ in this subsection. Let $\psi \in C_{\mathrm{pol}}^{4}(\R^{d};\R)$ and define the backward Kolmogorov function 
\begin{equation*}
      w(t,u) := \E\big[\psi(U(t,u;T))\big],
      \quad
      (t,u)\in[0,T]\times\R^{d}.
\end{equation*}
By Lemma \ref{lem:_coeff} and Proposition \ref{lem:flow_regularity_reduced}, together with the standard regularity theory for backward Kolmogorov equations (see, e.g., \cite[Theorem 5.6.1]{friedman1975stochastic}),  $w$ is a classical solution of 
\begin{equation}\label{eq:backward_kolmogorov_w}
\left\{\begin{aligned}
      &\partial_t w(t,u) + \mathcal L_t w(t,u) = 0,
      \quad
      (t,u) \in [0,T) \times \R^{d},
      \\
      &w(T,u) = \psi(u),
\end{aligned}\right.
\end{equation}
where $\mathcal{L}_{t}$ is the second-order differential operator acting on smooth test functions $\phi \colon \R^{d} \to \R$, defined by
\begin{equation}\label{eq:operatorLt}
      \mathcal{L}_{t} \phi(u)
      :=
      \langle f(t,u),D_{u}\phi(u) \rangle
      +
      \frac12 \operatorname{Tr}
      \big(g(t,u)g(t,u)^{\top} D_u^2\phi(u)\big).
\end{equation}
The following lemma collects the regularity estimates for $w$ needed in the local weak error analysis.

\begin{lemma}\label{lem:w_regularity}
      Suppose that Assumptions \ref{ass:At-structure}, \ref{ass:sdae-structure}, \ref{ass:coeff}, \ref{ass:FG-in-H} and \ref{ass:time_reg_merged} hold. Then there exist constants $C > 0$ and $q \geq 0$ such that for all $(t,u) \in [0,T] \times \R^{d}$,
      \begin{equation}\label{eq:w_spatial_growth}
            |D_u^j w(t,u)| \leq C(1+|u|^q), \quad j=0,1,2,3,4.
      \end{equation}
      Moreover, it holds that for all $(t,u) \in [0,T] \times \R^{d}$,
      \begin{equation}\label{eq:w_time_derivative_growth}
            |\partial_t D_u^j w(t,u)|
            \leq
            C(1+|u|^q), \quad j=0,1,2.
      \end{equation}
      Consequently, for any $s,t \in [0,T]$ and $u \in \R^{d}$,
      \begin{equation}\label{eq:w_time_lipschitz}
            |w(t,u)-w(s,u)| + |D_u w(t,u)-D_u w(s,u)| + |D_u^2w(t,u)-D_u^2w(s,u)|
            \leq
            C(1+|u|^q)|t-s|.
      \end{equation}
\end{lemma}

\begin{proof}
      By Proposition \ref{lem:flow_regularity_reduced}, the reduced SDE \eqref{eq:reducedsde} generates a stochastic flow $U(t,u;s)$, $0 \leq t \leq s \leq T$, which is four times continuously differentiable with respect to the initial value $u$. Moreover, for every $q \geq 2$, there exists a constant $C > 0$ such that for any $t \in [0,T], u \in \R^{d}$,
      \begin{align}\label{eq:estUDuj}
            \E\big[|U(t,u;T)|^{q}\big] \leq C(1+|u|^{q}),
            \quad
            \E\big[|D_u^jU(t,u;T)|^{q}\big] \leq C, \quad j = 1,\cdots,4.
      \end{align}

      We first prove the spatial derivative estimate. Since $\psi \in C_{\mathrm{pol}}^{4}(\R^{d};\R)$, there exist constants $C > 0$ and $\chi \geq 0$ such that $|D^j\psi(z)| \leq C(1+|z|^\chi)$ for all $z \in \R^{d}$ and $j = 0,1,2,3,4$. It follows from \eqref{eq:estUDuj} that for the case $j = 0$,
      \begin{align*}
            |D_u^0 w(t,u)| = |w(t,u)|
            \leq
            \E\big[|\psi(U(t,u;T))|\big]
            \leq
            C\big(1 + \E\big[|U(t,u;T)|^{\chi}\big]\big)
            \leq
            C\big(1 + |u|^{\chi}\big).
      \end{align*}
      To estimate $D_u^j w(t,u)$ for $j = 1,2,3,4$, we need to justify the interchange of differentiation and expectation. For every direction $\xi \in \R^{d}$, the difference quotient of $w$ satisfies
      \begin{equation*}
            \frac{w(t,u+\varepsilon \xi)-w(t,u)}{\varepsilon}
            =
            \E\left[\frac{\psi(U(t,u+\varepsilon \xi;T))
            - \psi(U(t,u;T))}{\varepsilon}\right].
      \end{equation*}
      Owing to the differentiability of $U(t,u;T)$ with respect to $u$ in Proposition \ref{lem:flow_regularity_reduced}, we have 
      \begin{equation*}
            \left.\frac{d}{d\varepsilon}\psi\bigl(U(t,u+\varepsilon\xi;T)\bigr)\right|_{\varepsilon=0}
            =
            D\psi\bigl(U(t,u;T)\bigr)\bigl[D_uU(t,u;T)[\xi]\bigr], 
            \quad \text{$\P$-a.s.}
      \end{equation*}
      Equivalently,
      \begin{equation*}
            Q_{\varepsilon}
            :=
            \frac{\psi\bigl(U(t,u+\varepsilon\xi;T)\bigr)
            - \psi\bigl(U(t,u;T)\bigr)}{\varepsilon}
            \longrightarrow
            D\psi\bigl(U(t,u;T)\bigr)\bigl[D_uU(t,u;T)[\xi]\bigr]
      \end{equation*}
      almost surely as $\varepsilon \to 0$. We shall show that, for $\varepsilon_0>0$ sufficiently small, the family $
      \{Q_\varepsilon: 0< |\varepsilon| \le \varepsilon_0\}$ is uniformly integrable. Choose a compact neighbourhood $K$ of $u$, for instance $K:=\overline{B(u,1)}$. Taking $\varepsilon_0 > 0$ sufficiently small (e.g., $\varepsilon_0 \leq |\xi|^{-1}$), we may assume that
      \begin{equation*}
            u+\rho\varepsilon\xi\in K,
            \quad 0<|\varepsilon|\le\varepsilon_0, \quad \rho\in[0,1].
      \end{equation*}
      By Proposition \ref{lem:flow_regularity_reduced}, the map $v \mapsto U(t,v;T)$ is differentiable with respect to the initial value. Therefore, for fixed $\omega$, applying the chain rule along the line segment $v_\rho := u + \rho\varepsilon\xi$ gives
      \begin{align*}
            |Q_{\varepsilon}|
            =&~
            \left|\int_0^1 D\psi(U(t,u+\rho\varepsilon\xi;T))
            \big[D_uU(t,u+\rho\varepsilon\xi;T)[\xi]\big]\,d\rho\right|
            \\\leq&~
            |\xi| \int_0^1 |D\psi(U(t,u+\rho\varepsilon\xi;T))|
            |D_uU(t,u+\rho\varepsilon\xi;T)|\,d\rho
            \\\leq&~
            C|\xi| \int_0^1 \bigl(1+|U(t,u+\rho\varepsilon\xi;T)|^\chi\bigr)
            |D_uU(t,u+\rho\varepsilon\xi;T)|\,d\rho.
      \end{align*}
      Now choose some $\gamma > 1$, for example $\gamma = 2$. By the H\"{o}lder inequality, one has
      \begin{align*}
            \E\big[|Q_\varepsilon|^{\gamma}\big]
            \leq&~
            C|\xi|^\gamma \int_0^1 \E\left[
            \bigl(1+|U(t,u+\rho\varepsilon\xi;T)|^\chi\bigr)^\gamma
            |D_uU(t,u+\rho\varepsilon\xi;T)|^\gamma \right] \,d\rho
            \\\leq&~
            C|\xi|^\gamma \int_0^1 \big(\E\big[
            \bigl(1+|U(t,u+\rho\varepsilon\xi;T)|^\chi\bigr)^{2\gamma} 
            \big]\big)^{\frac{1}{2}}
            \big(\E\left[|D_uU(t,u+\rho\varepsilon\xi;T)|^{2\gamma} 
            \right]\big)^{\frac{1}{2}} \,d\rho.
      \end{align*}
      Since $K$ is compact, \eqref{eq:moment_exact_solution_all_q} and \eqref{eq:ass9-derivative-weighted} indicate that
      \begin{align*}
            \sup_{v\in K}\E\big[|U(t,v;T)|^{2\gamma\chi}\big]        
            \leq
            \sup_{v\in K}C(1+|v|^{2\gamma\chi}) 
            < +\infty,
            \quad
            \sup_{v \in K}\E\big[|D_uU(t,v;T)|^{2\gamma}\big] < +\infty.
      \end{align*}
      Together with $u + \rho\varepsilon\xi \in K$ for all $0 < |\varepsilon| \leq \varepsilon_{0}$ and $\rho \in [0,1]$, we obtain
      \begin{align*}
            \sup_{0 < |\varepsilon| \leq \varepsilon_{0}} 
            \E\big[|Q_\varepsilon|^{\gamma}\big]
            \leq
            C_{K}|\xi|^{\gamma}
            <
            +\infty,
      \end{align*}
      which together with
      $
            \E\big[|Q_\varepsilon|\mathbf{1}_{\{|Q_\varepsilon|>R\}}\big]
            \leq
            \E\big[|Q_\varepsilon|^\gamma\big] R^{1-\gamma}  
      $ 
      for any $R > 0$ implies
      \begin{align*}
            \lim_{R \to \infty} \sup_{0 < |\varepsilon| \leq \varepsilon_{0}} 
            \mathbb{E} \left[|Q_\varepsilon|\mathbf{1}_{\{|Q_\varepsilon|>R\}}\right]
            =
            0.
      \end{align*}
      Thus the family $\{Q_\varepsilon: 0< |\varepsilon| \le \varepsilon_0\}$ is uniformly integrable. It follows from the Vitali convergence theorem that
      \begin{equation*}
            D_uw(t,u)[\xi]
            =
            \E\left[D\psi(U(t,u;T))D_uU(t,u;T)[\xi] \right].
      \end{equation*}
      The same arguments apply to the higher-order directional derivatives. Indeed, for $j = 2,3,4$, the multivariate Fa\`a di Bruno formula expresses $D_u^j(\psi(U(t,u;T)))$ as a finite sum of terms of the form
      \begin{equation*}
            D^k\psi(U(t,u;T))\Big[D_u^{\ell_1}U(t,u;T),
            \cdots, D_u^{\ell_k}U(t,u;T)\Big],
            \quad
            \ell_1 + \cdots + \ell_k = j.
      \end{equation*}
      Each such term is integrable and locally uniformly integrable by the H\"older inequality, the polynomial growth of $D^k\psi$, and the moment bounds for the flow derivatives. Therefore, $D_u^j$ and $\E$ can be interchanged for $j = 1,\cdots,4$.
      %
      %
      Differentiating $w(t,u)=\E[\psi(U(t,u;T))]$ with respect to $u$ and using the multivariate Fa\`a di Bruno formula, we obtain
      \begin{align*}
            D_u^j w(t,u)[\xi_1,\cdots,\xi_j]
            =
            \E\bigg[\sum_{\pi\in\Pi_j}D^{|\pi|}\psi(U)
            \Big[J^{(|B_1|)}(t,u;T)[\xi_{B_1}],
            \cdots,
            J^{(|B_{|\pi|}|)}(t,u;T)[\xi_{B_{|\pi|}}]\Big]\bigg],
      \end{align*}
      where $\Pi_j$ denotes the set of all partitions $\pi = \{B_1,\cdots,B_{|\pi|}\}$ of $\{1,\cdots,j\}$. For a block $B = \{i_1,\cdots,i_\ell\}$, we use the notation $J^{(|B|)}(t,u;T)[\xi_B] := D_u^{\ell}U(t,u;T)[\xi_{i_1},\cdots,\xi_{i_\ell}]$. Since the number of partitions of $\{1,\cdots,j\}$ is finite, using the polynomial growth of $D^k\psi$, the H\"older inequality, the moment estimate of $U(t,u;T)$ in \eqref{eq:estUDuj}, and the derivative moment bounds for $J^{(\ell)}$, we get
      \begin{align*}
            |D_u^j w(t,u)[\xi_1,\cdots,\xi_j]| \leq C(1+|u|^q).
      \end{align*}
      Taking the supremum over $|\xi_1|,\cdots,|\xi_j| \leq 1$ yields $|D_u^j w(t,u)| \leq C(1+|u|^q)$ for $j = 1,2,3,4$. Together with the case $j = 0$, this proves \eqref{eq:w_spatial_growth}.

      We next prove the time-derivative estimates. Owing to \eqref{eq:backward_kolmogorov_w},  \eqref{eq:operatorLt}, \eqref{eq:w_spatial_growth} and the linear growth of $f$ and $g$ in Lemma \ref{lem:_coeff}, we have
      \begin{align}\label{eq:partialtD0wtu}
            |\partial_t w(t,u)|
            \leq
            |f(t,u)|\,|D_u w(t,u)| + \frac12 |g(t,u)|^2\,|D_u^2 w(t,u)|
            \leq
            C(1+|u|^q).
      \end{align}
      It remains to estimate $\partial_tD_uw$ and $\partial_tD_u^2w$. Since $w$ is a classical solution of the backward Kolmogorov equation \eqref{eq:backward_kolmogorov_w}, differentiating $\partial_t w(t,u) = -\mathcal L_t w(t,u)$ with respect to $u$ gives
      \begin{align}\label{eq:partialtDuj}
            \partial_tD_u^j w(t,u)
            =
            -D_u^j(\mathcal L_t w)(t,u), \quad j = 1,2.
      \end{align}
      To compute $D_u^j(\mathcal L_t w)$ explicitly, writting $g^k(t,u)$ for the $k$-th column of $g(t,u), k = 1,\cdots,m$ yields
      \begin{align*}
            \mathcal L_t w(t,u)
            =
            D_uw(t,u)[f(t,u)] + \frac12\sum_{k=1}^{m}D_u^2w(t,u)[g^k(t,u),g^k(t,u)].
      \end{align*}
      For any $\xi,\eta \in \R^{d}$, the first and the second spatial derivatives are given by
      \begin{align*}
            D_u(\mathcal L_t w)(t,u)[\xi]
            =&~
            D_u^2w(t,u)[\xi,f(t,u)] + D_uw(t,u)[D_uf(t,u)[\xi]]
            \\&~+
            \frac12\sum_{k=1}^m D_u^3w(t,u)[\xi,g^k(t,u),g^k(t,u)]
            +
            \sum_{k=1}^m D_u^2w(t,u)[D_ug^k(t,u)[\xi],g^k(t,u)],
      \end{align*}
      and
      \begin{align*}
            &~D_u^2(\mathcal L_t w)(t,u)[\xi,\eta]
            \\=&~
            D_u^3w(t,u)[\xi,\eta,f(t,u)] + D_u^2w(t,u)[\xi,D_uf(t,u)[\eta]]
            +
            D_u^2w(t,u)[\eta,D_uf(t,u)[\xi]]
            \\&~+
            D_uw(t,u)[D_u^2f(t,u)[\xi,\eta]]
            +
            \frac12\sum_{k=1}^{m} D_{u}^{4}w(t,u)[\xi,\eta,g^k(t,u),g^k(t,u)]
            \\&~+
            \sum_{k=1}^m D_u^3w(t,u)[\xi,D_ug^k(t,u)[\eta],g^k(t,u)]
            +
            \sum_{k=1}^m D_u^3w(t,u)[\eta,D_ug^k(t,u)[\xi],g^k(t,u)]
            \\&~+
            \sum_{k=1}^m D_u^2w(t,u)[D_u^2g^k(t,u)[\xi,\eta],g^k(t,u)]
            +
            \sum_{k=1}^m D_u^2w(t,u)[D_ug^k(t,u)[\xi],D_ug^k(t,u)[\eta]].
      \end{align*}
      Thanks to Lemma \ref{lem:_coeff}, there exists a constant $C > 0$ such that for any $\xi,\eta \in \R^{d}$,
      \begin{equation*}
            \left|D_u(\mathcal L_t w)(t,u)[\xi]\right| \leq C(1+|u|^q)|\xi|,
            \quad
            \left|D_u^2(\mathcal L_t w)(t,u)[\xi,\eta]\right| \leq C(1+|u|^q)|\xi|\,|\eta|.
      \end{equation*}
      Taking the supremum over $|\xi|,|\eta| \leq 1$ and using \eqref{eq:partialtDuj} yield
      \begin{equation*}
            |\partial_t D_u w(t,u)| + |\partial_t D_u^2 w(t,u)|
            =
            |D_u(\mathcal L_t w)(t,u)| + |D_u^2(\mathcal L_t w)(t,u)|
            \leq
            C(1+|u|^q),
      \end{equation*}
      which together with \eqref{eq:partialtD0wtu} proves \eqref{eq:w_time_derivative_growth}.

      Finally, for $j=0,1,2$ and $s,t\in[0,T]$, the fundamental theorem of calculus gives
      \begin{equation*}
            D_u^j w(t,u)-D_u^j w(s,u)
            =
            \int_s^t \partial_rD_u^j w(r,u)\,dr.
      \end{equation*}
      It follows from \eqref{eq:w_time_derivative_growth} that
      \begin{equation*}
            |D_u^j w(t,u)-D_u^j w(s,u)|
            \leq
            C(1+|u|^q)|t-s|, \quad j=0,1,2,
      \end{equation*}
      which yields \eqref{eq:w_time_lipschitz}. The proof is complete.
\end{proof}

\subsubsection{One-step weak local error}
We now verify the one-step weak local error condition in Theorem \ref{thm:weak_convergence_sdae}. Inspired by the frozen reference step argument in \cite{chen2026weak}, we insert a frozen Euler reference step between the exact reduced flow and the differential update induced by the SDAE-level stochastic theta method. This decomposes the local weak error into a standard Brownian weak expansion part and an implicit correction part caused by the drift evaluation at the new time level. Let us first derive the standard one-step weak expansion for the following frozen Euler step
\begin{align*}
      \bar{u}(t,u;t+h)
      :=
      u + hf(t,u) + g(t,u)\Delta{W_{t,h}}, 
      \quad \Delta{W_{t,h}} := W_{t+h}-W_{t},
      \quad t \in [0,T-h], u \in \R^{d}.
\end{align*}

\begin{lemma}\label{lem:frozen_step_taylor}
      Suppose that Assumptions \ref{ass:At-structure}, \ref{ass:sdae-structure}, \ref{ass:coeff}, \ref{ass:FG-in-H} and \ref{ass:time_reg_merged} hold, and let $h \in (0,h_{0}]$. Then there exist constants $C > 0$ and $q \geq 0$ such that for any $t \in [0,T-h]$ and $u \in \R^{d}$,
      \begin{equation}\label{eq:frozen_step_taylor}
            \left|\E\bigl[w(t+h,\bar u(t,u;t+h))\bigr] - w(t+h,u) - h\mathcal L_t w(t+h,u)\right|
            \leq
            C(1+|u|^q)h^2.
      \end{equation}
\end{lemma}

\begin{proof}
      Since $\Delta{W_{t,h}}$ is independent of $\mathcal F_t$ and satisfies
      \begin{align*}
            \E[\Delta{W_{t,h}}]=0,
            \quad
            \E[\Delta{W_{t,h}}^{k}\Delta{W_{t,h}}^{\ell}] = h\delta_{k\ell},
            \quad
            \E\big[|\Delta{W_{t,h}}|^{q}\big] \leq C_{r} h^{\frac{q}{2}},
            \quad r \geq 2,
      \end{align*}
      we use the linear growth of $f$ and $g$ as well as $h \leq h_{0} < 1$ to show that for every $q \geq 2$,
      \begin{align}\label{eq:eta-moment-frozen}
            \E\big[|hf(t,u) + g(t,u)\Delta{W_{t,h}}|^{q}\big]
            \leq
            Ch^{q}|f(t,u)|^{q} + C|g(t,u)|^{q}\E\big[|\Delta{W_{t,h}}|^{q}\big]
            \leq
            C(1+|u|^{q})h^{\frac{q}{2}}.
      \end{align}
      For simplicity, we write $\eta := hf(t,u) + g(t,u)\Delta{W_{t,h}}$ and $\phi(x) := w(t+h,x), x \in \R^{d}$.  Applying $\bar{u}(t,u;t+h) = u+\eta$ and the Taylor formula with integral remainder leads to
      \begin{align*}
            \phi(\bar{u}(t,u;t+h))
            =
            \phi(u) + D\phi(u)[\eta]
            +
            \frac12D^2\phi(u)[\eta,\eta]
            +
            \frac16D^3\phi(u)[\eta,\eta,\eta] + R_4,
      \end{align*}
      where $R_{4} := \frac{1}{6}\int_{0}^{1}(1-\rho)^{3} D^{4}\phi(u+\rho\eta)[\eta,\eta,\eta,\eta]\,d\rho$.
      First, $\E[D\phi(u)[\eta]] = hD\phi(u)[f(t,u)]$. Second,
      \begin{align*}
            \E\big[D^2\phi(u)[\eta,\eta]\big]
            =
            h\sum_{k=1}^mD^2\phi(u)[g^{k}(t,u),g^{k}(t,u)]
            +
            h^2D^2\phi(u)[f(t,u),f(t,u)]
      \end{align*}
      with $\left|h^2D^2\phi(u)[a,a]\right| \leq C(1+|u|^q)h^2$ because of the growth estimate of $D^2w$ in Lemma \ref{lem:w_regularity} and the linear growth of $f$. For the third-order term, using the vanishing of odd moments of the centered Gaussian increment, we have
      \begin{align*}
            \E\big[D^3\phi(u)[\eta,\eta,\eta]\big]
            =
            h^3D^3\phi(u)[f(t,u),f(t,u),f(t,u)]
            +
            3h^2\sum_{k=1}^mD^3\phi(u)[f(t,u),g^{k}(t,u),g^{k}(t,u)]
      \end{align*}
      with $\big|\E\big[D^{3}\phi(u)[\eta,\eta,\eta]\big]\big| \leq C(1+|u|^q)h^{2}$, which is due to the growth estimate of $D^3w$ and the linear growth of $f$ and $g$. For the remainder, Lemma \ref{lem:w_regularity} gives $|D^4\phi(x)| = |D_u^4w(t+h,x)| \leq C(1+|x|^q)$, which together with \eqref{eq:eta-moment-frozen} yields
      \begin{align*}
            \E\big[|R_4|\big]
            \leq
            C\E\left[(1+|u+\rho\eta|^q)|\eta|^4\right]  
            \leq
            C(1+|u|^q)\E\big[|\eta|^4\big] + C\E\big[|\eta|^{q+4}\big]
            \leq
            C(1+|u|^q)h^{2}.
      \end{align*}
      Combining the preceding estimates, we obtain
      \begin{align*}
            \E\big[\phi(\bar{u}(t,u;t+h))\big]
            &=
            \phi(u) + hD\phi(u)[f(t,u)]
            +
            \frac h2\sum_{k=1}^mD^2\phi(u)[g^{k}(t,u),g^{k}(t,u)] + \mathcal R_h
      \end{align*}
      with $|\mathcal{R}_{h}| \leq C(1 + |u|^{q})h^{2}$.
      It follows from
      \begin{align*}
            D\phi(u)[f(t,u)] + \frac12\sum_{k=1}^mD^2\phi(u)[g^{k}(t,u),g^{k}(t,u)]
            =
            \mathcal L_t w(t+h,u)
      \end{align*}
      that
      $\E\bigl[w(t+h,\bar u(t,u;t+h))\bigr] = w(t+h,u) + h\mathcal L_t w(t+h,u) + \mathcal R_h$, which proves \eqref{eq:frozen_step_taylor}.
\end{proof}

Combining the preceding one-step expansion and the backward Kolmogorov equation \eqref{eq:backward_kolmogorov_w}, we obtain the weak error estimate between the exact reduced flow and the frozen Euler step.

\begin{proposition}\label{prop:frozen_step_weak_error}
      Suppose that Assumptions \ref{ass:At-structure}, \ref{ass:sdae-structure}, \ref{ass:coeff}, \ref{ass:FG-in-H} and \ref{ass:time_reg_merged} hold, and let $h \in (0,h_{0}]$. Then there exist constants $C > 0$ and $q \geq 0$, independent of $t,h$, and $u$, such that
      \begin{equation*}
            \left|\E\bigl[w(t+h,U(t,u;t+h))\bigr]
            - \E\bigl[w(t+h,\bar u(t,u;t+h))\bigr]\right|
            \leq
            C(1+|u|^q)h^2.
      \end{equation*}
\end{proposition}

\begin{proof}
      By the semigroup property of the Kolmogorov function \eqref{eq:w_semigroup_B} with $s = t+h$, we have
      \begin{equation}\label{eq:exact-semigroup-local}
            \E\bigl[w(t+h,U(t,u;t+h))\bigr] = w(t,u).
      \end{equation}
      Thus it remains to compare $w(t,u)$ with the first-order frozen expansion at time $t+h$. From  \eqref{eq:backward_kolmogorov_w},
      we get $w(t+h,u) - w(t,u) = -\int_t^{t+h}\mathcal L_s w(s,u)\,ds$,
      which implies
      \begin{align}\label{eq:kolmogorov-defect}
            &~w(t+h,u)+h\mathcal L_t w(t+h,u)-w(t,u) \notag
            =
            \int_t^{t+h} \big(\mathcal L_t w(t+h,u)-\mathcal L_s w(s,u)\big)\,ds \notag
            \\=&~
            \int_t^{t+h} \big(\mathcal L_t w(t+h,u)-\mathcal L_s w(t+h,u)\big)
            +
            \big(\mathcal L_s w(t+h,u)-\mathcal L_s w(s,u)\big)\,ds.
      \end{align}
      By the fundamental theorem of calculus and the time-regularity estimates for $f$ and $g$ in \eqref{eq:f-g-time-derivative-growth}, after possibly increasing $q$, we have
      \begin{equation}\label{eq:fg-time-difference-for-frozen}
            |f(s,u)-f(t,u)| + |g(s,u)-g(t,u)|
            \leq
            C(1+|u|^q)|s-t|.
      \end{equation}
      which together with the linear growth of $g$ gives 
      \begin{equation*}
           |g(s,u)g(s,u)^\top-g(t,u)g(t,u)^\top| 
           \leq
           |g(s,u)-g(t,u)|\bigl(|g(s,u)|+|g(t,u)|\bigr) 
           \leq
           C(1+|u|^q)|s-t|.
      \end{equation*}
      It follows from the definition of $\mathcal L_t$, \eqref{eq:w_spatial_growth} and \eqref{eq:fg-time-difference-for-frozen} that
      \begin{align}\label{eq:L-coefficient-difference}
            \big|\mathcal L_t w(t+h,u)-\mathcal L_s w(t+h,u)\big| \notag
            =&~
            \Big|\big\langle f(t,u)-f(s,u),D_uw(t+h,u)\big\rangle \notag
            \\&~+
            \frac12\operatorname{Tr}\!\big(\big(g(t,u)g(t,u)^{\top}
            - g(s,u)g(s,u)^{\top}\big)D_u^2w(t+h,u)\big)\Big| \notag
            \\\leq&~
            \frac12 |g(t,u)g(t,u)^\top-g(s,u)g(s,u)^\top|\,|D_u^2w(t+h,u)| \notag
            \\&+
            |f(t,u)-f(s,u)|\,|D_uw(t+h,u)| \notag
            \\\leq&~
            C(1+|u|^q)|s-t|.
      \end{align}
      Applying the definition of $\mathcal L_s$, the linear growth of $f$ and $g$ and \eqref{eq:w_time_lipschitz} leads to
      \begin{align}\label{eq:L-w-time-difference}
            \big|\mathcal L_s w(t+h,u) - \mathcal L_s w(s,u) \big| \notag 
            =&~
            \Big|\left\langle f(s,u), D_u w(t+h,u)-D_u w(s,u) \right\rangle \notag
            \\&~
            + \frac12 \operatorname{Tr}\!\left(g(s,u)g(s,u)^\top
            \bigl(D_u^2w(t+h,u)-D_u^2w(s,u)\bigr)\right)\Big| \notag
            \\\leq&~
            |f(s,u)|\,|D_u w(t+h,u)-D_u w(s,u)| \notag
            \\&~
            + \frac12 |g(s,u)|^2 \,|D_u^2w(t+h,u)-D_u^2w(s,u)| \notag
            \\\leq&~
            C(1+|u|^q)|t+h-s|.
      \end{align}
      Combining \eqref{eq:kolmogorov-defect}, \eqref{eq:L-coefficient-difference}, and \eqref{eq:L-w-time-difference}, one gets
      \begin{equation}\label{eq:kolmogorov-defect-bound}
            \left|w(t+h,u) + h\mathcal L_t w(t+h,u)-w(t,u)\right| 
            \leq
            C(1+|u|^q)h^2.
      \end{equation}
      Due to \eqref{eq:exact-semigroup-local}, \eqref{eq:kolmogorov-defect-bound} and Lemma \ref{lem:frozen_step_taylor}, we finally obtain
      \begin{align*}
            &~\left|\E\big[w(t+h,U(t,u;t+h))\big]
            - \E\big[w(t+h,\bar u(t,u;t+h))\big]\right|  
            \\=&~
            \left|w(t,u) - \E\bigl[w(t+h,\bar u(t,u;t+h))\bigr]\right|  
            \\\leq&~
            \left|w(t,u)-w(t+h,u)-h\mathcal L_t w(t+h,u)\right|  
            \\&~
            +\left|\E\bigl[w(t+h,\bar u(t,u;t+h))\bigr]
            - w(t+h,u) - h\mathcal L_t w(t+h,u)\right|
            \\\leq&~
            C(1+|u|^q)h^2,
      \end{align*}
      which proves the required and thus complete the proof.
\end{proof}

It remains to estimate the weak contribution of the implicit correction. 
For any $h \in (0, h_{0}]$, set
\begin{align*}
      Y_{t} := u + \widehat{V}(t,u),
      \quad
      \Delta{W_{t,h}} := W_{t+h} - W_{t},
      \quad t \in [0,T-h], u \in \R^{d}.          
\end{align*}
Let $\widetilde{u} = \widetilde{u}(t,u;t+h)$ be the differential component generated by the local stochastic theta step
\begin{align*}
      \widetilde{u}
      =
      u + h(1-\theta)f(t,u)
      + 
      h\theta A_t^{-}F(t+h,\widetilde Y_{t+h}) + g(t,u)\Delta W_{t,h},
\end{align*}
where $\widetilde{Y}_{t+h} := \widetilde{u} + \widehat{V}(t+h,\widetilde{u})$. The following proposition estimates the weak difference between this local differential update and the frozen Euler reference step.

\begin{proposition}\label{lem:implicit_correction_weak}
      Suppose that Assumptions \ref{ass:At-structure}, \ref{ass:sdae-structure}, \ref{ass:coeff}, \ref{ass:FG-in-H} and \ref{ass:time_reg_merged} hold, and let $\theta \in (0,1]$, $h \in (0,h_{0}]$. 
      Then there exist constants $C > 0$ and $q \geq 0$, independent of $h,t,u$, such that
      \begin{align*}
            \left| \E\big[w(t+h,\widetilde{u}(t,u;t+h))\big]
            - \E\big[w(t+h,\bar{u}(t,u;t+h))\big] \right|
            \leq
            C(1+|u|^q)h^2.
      \end{align*}
\end{proposition}

\begin{proof}
      For simplicity, write $\bar{u} := \bar{u}(t,u;t+h)$ and $\delta := \widetilde{u} - \bar{u}$. Using $f(t,u) = A_t^{-}F(t,Y_t)$ gives 
      \begin{equation}\label{eq:delta-implicit-correction}
            \delta
            =
            h\theta A_t^{-}\big(F(t+h,\widetilde{Y}_{t+h})-F(t,Y_t)\big).
      \end{equation}
      Recall that the local stochastic theta step is simply one step of the full method on $[t,t+h]$, with
      \begin{align*}
            t_n=t,\quad t_{n+1}=t+h,\quad
            Y_n=Y_t:=u+\widehat{V}(t,u),
            \quad Y_{n+1}=\widetilde Y_{t+h}.     
      \end{align*}
      Moreover, in the present local estimate, $u$ is fixed and deterministic. Hence $Y_t = u + \widehat{V}(t,u)$ is also deterministic, and the only randomness comes from the Brownian increment $\Delta{W_{t,h}}$. Applying the increment estimates in Proposition \ref{lem:moment_Y} to this single step shows that for every $q \geq 2$,
      \begin{align}\label{eq:local-Y-increment-Lr}
            \E\big[|\widetilde{Y}_{t+h} - Y_{t}|^{q}\big]
            \leq
            Ch^{\frac{q}{2}}(1 + |Y_t|^{q})
            \leq
            Ch^{\frac{q}{2}}(1 + |u|^{q}),
      \end{align}
      where we have used $|Y_t| \leq |u| + |\widehat{V}(t,u)| \leq C(1+|u|)$  because of the linear growth of $\widehat{V}$. Similarly, Proposition \ref{lem:moment_Y}, applied to the  deterministic initial value $Y_{t}$, further gives
      \begin{equation}\label{eq:local-Y-increment-mean}
            \big|\E\big[\widetilde{Y}_{t+h} - Y_{t}\big]\big|
            \leq
            Ch(1+|Y_t|^2)
            \leq
            Ch(1+|u|^2).
      \end{equation}

      Let us derive two estimates for $\delta$. Using \eqref{eq:delta-implicit-correction}, \eqref{eq:local-Y-increment-Lr}, the boundedness of $A_t^{-}$ and the Lipschitz continuity of $F$ in the space variable and its time-regularity, we obtain that for every $q \geq 2$,
      \begin{align}\label{eq:delta-Lr}
            \E\big[|\delta|^{q}\big]
            \leq&~
            Ch^{q}\E\big[|F(t+h,\widetilde{Y}_{t+h}) - F(t,Y_t)|^{q}\big] \notag
            \\\leq&~
            Ch^{q}\E\big[|F(t+h,\widetilde{Y}_{t+h}) - F(t+h,Y_t)|^{q}\big] \notag
            \\&~+
            Ch^{q}\E\big[|F(t+h,Y_t) - F(t,Y_t)|^{q}\big] \notag
            \\\leq&~
            Ch^{q}\E\big[|\widetilde{Y}_{t+h} - Y_{t}\big|^{q}\big] \notag
            +
            Ch^{2q}(1+|Y_{t}|^{q})  
            \\\leq&~
            C(1 + |u|^{q})h^{\frac{3q}{2}}.
      \end{align}
      Besides, the Taylor formula in the spatial variable implies
      \begin{align*}
            F(t+h,\widetilde{Y}_{t+h}) - F(t+h,Y_t)
            =
            D_xF(t+h,Y_t)(\widetilde{Y}_{t+h}-Y_t) + R_F,
      \end{align*}
      where $R_{F} := \int_0^1(1-\rho) D_x^2F\bigl(t+h,Y_t+\rho(\widetilde{Y}_{t+h}-Y_t)\bigr)
      [\widetilde{Y}_{t+h} - Y_t,\widetilde{Y}_{t+h}-Y_t]\,d\rho$. By \eqref{eq:delta-implicit-correction}, \eqref{eq:local-Y-increment-Lr}, \eqref{eq:local-Y-increment-mean}, the boundedness of $D_{x}^{2}F$ in Assumption \ref{ass:FG-in-H} and Assumption \ref{ass:time_reg_merged}, we have
      \begin{align}\label{eq:delta-mean}
            |\E[\delta]|
            \leq&~
            Ch\big|D_xF(t+h,Y_t)\big|\,\big|\E\big[\widetilde{Y}_{t+h}-{Y}_t)\big]\big| 
            +
            Ch\E\big[|R_{F}|\big] + Ch\big|F(t+h,Y_t) - F(t,Y_t)\big| \notag 
            \\\leq&~
            Ch\big|\E\big[\widetilde{Y}_{t+h} - Y_t\big]\big|
            +
            Ch\E\big[|\widetilde{Y}_{t+h}-Y_t|^{2}\big] + Ch^2(1 + |Y_{t}|) \notag
            \\\leq&~
            C(1+|u|^2)h^2.
      \end{align}

      We now compare $w(t+h,\widetilde{u})$ and $w(t+h,\bar{u})$. The Taylor formula for $w(t+h,\cdot)$ gives
      \begin{align*}
            w(t+h,\widetilde{u}) - w(t+h,\bar{u})
            =
            D_uw(t+h,\bar{u})[\delta] + R_w
      \end{align*}
      with $R_{w} = \int_0^1(1-\rho) D_u^2w(t+h,\bar{u} + \rho\delta)[\delta,\delta]\,d\rho$. Utilizing the decomposition
      \begin{align*}
            D_uw(t+h,\bar u)[\delta]
            =
            D_uw(t+h,u)[\delta] + \bigl(D_uw(t+h,\bar u)-D_uw(t+h,u)\bigr)[\delta]
      \end{align*}
      leads to
      \begin{align}\label{eq:wtildeuminbaru}
            \E\big[w(t+h,\widetilde{u})\big] - \E\big[w(t+h,\bar{u})\big] 
            =&~
            \E\big[D_uw(t+h,u)[\delta]\big] + \E[R_w] \notag
            \\&~+
            \E\big[\big(D_uw(t+h,\bar{u}) - D_uw(t+h,u)\big)[\delta]\big].
      \end{align}
      Let us estimate these three terms separately. For $\E\big[D_uw(t+h,u)[\delta]\big]$,  \eqref{eq:w_spatial_growth} and \eqref{eq:delta-mean} show that
      \begin{equation}\label{eq:first-weak-correction-term}
            \big|\E\big[D_uw(t+h,u)[\delta]\big]\big|
            =
            \big|D_{u}w(t+h,u)\E\big[\delta\big]\big|  
            \leq
            C(1 + |u|^{q})\big|\E\big[\delta\big]\big|  
            \leq
            C(1+|u|^q)h^{2}.
      \end{equation}
      Concerning $\E\big[\big(D_uw(t+h,\bar{u}) - D_uw(t+h,u)\big)[\delta]\big]$, applying \eqref{eq:w_spatial_growth} and the Taylor formula with integral remainder to $x \mapsto D_{u} w(t+h,x)$ leads to
      \begin{align*}
            |D_u w(t+h,\bar u)-D_u w(t+h,u)|
            \leq&~
            |\bar{u}-u|\int_0^1 |D_u^2w(t+h,u+\rho\Delta\bar u)|\,d\rho
            \\\leq&~
            C\bigl(1+|u|^q+|\bar u-u|^q\bigr)|\bar u-u|.
      \end{align*}
      In combination with
      \begin{align}\label{estbaruminur}
            \E\big[|\bar{u}-u|^{q}\big]
            =
            \E\big[|hf(t,u) + g(t,u)\Delta{W_{t,h}}|^{q}\big]
            \leq
            C(1+|u|^{q})h^{\frac{q}{2}}, \quad q \geq 2,
     \end{align}
     one can use the H\"{o}lder inequality and \eqref{eq:delta-Lr}, after increasing $q$ if necessary, to get
     \begin{align}\label{eq:second-weak-correction-term}
           &~\left|\E\big[\big(D_uw(t+h,\bar{u})-D_uw(t+h,u)\big)[\delta]\big]\right| \notag 
           \\\leq&~
           \E\big[|D_uw(t+h,\bar{u})-D_uw(t+h,u)|\,|\delta|\big] \notag
           \\\leq&~
           C(1+|u|^q)\E\big[|\bar{u}-u|\,|\delta|\big]
           +
           C\E\big[|\bar{u}-u|^{q+1}\,|\delta|\big] \notag
           \\\leq&~
           C(1+|u|^q)\big(\E\big[|\bar{u}-u|^{2}\big]\big)^{\frac{1}{2}}
           \big(\E\big[|\delta|^{2}\big]\big)^{\frac{1}{2}}
           +
           C\big(\E\big[|\bar{u}-u|^{2(q+1)}\big]\big)^{\frac{1}{2}}
           \big(\E\big[|\delta|^{2}\big]\big)^{\frac{1}{2}} \notag
           \\\leq&~
           C(1+|u|^q)h^{\frac{1}{2}}h^{\frac{3}{2}}
           +
           C(1+|u|^q)h^{\frac{1+q}{2}}h^{\frac{3}{2}} \notag
           \\\leq&~
           C(1+|u|^{q})h^{2}.
      \end{align}
      To estimate $\E[R_w]$, we use the H\"{o}lder inequality, \eqref{eq:w_spatial_growth} and \eqref{eq:delta-Lr} to deduce that
      \begin{align}\label{eq:remainder-weak-correction}
            \big|\E\big[R_{w}\big]\big|
            \leq&~
            C\E\left[\int_0^1(1-\rho)
            \big(1 + |\bar{u} + \rho\delta|^{q}\big)|\delta|^{2}\,d\rho\right] \notag 
            \\\leq&~
            C\E\left[\big(1 + |\bar{u}|^{q} + |\delta|^{q}\big)|\delta|^{2}\right] \notag
            \\\leq&~
            C\E\left[\big(1 + |u|^{q} + |\bar{u}-u|^{q} 
            + |\delta|^{q}\big)|\delta|^{2}\right] \notag
            \\\leq&~
            C\big(1 + |u|^{q}\big)\E\left[|\delta|^{2}\right]
            +
            C\big(\E\left[|\bar{u}-u|^{2q}\right]\big)^{\frac{1}{2}}
            \big(\E\left[|\delta|^{4}\right]\big)^{\frac{1}{2}}
            +
            C\E\left[|\delta|^{q+2}\right] \notag
            \\\leq&~
            C(1+|u|^{q})h^{3}.
      \end{align}
      Inserting \eqref{eq:first-weak-correction-term}, \eqref{eq:second-weak-correction-term} and \eqref{eq:remainder-weak-correction} into \eqref{eq:wtildeuminbaru} yields the desired result
      and thus complete the proof.
\end{proof}

From the frozen Euler estimate and the implicit correction estimate, we obtain the following one-step weak local error bound.
\begin{proposition}\label{prop:verify_one_step_w_theta}
      Suppose that Assumptions \ref{ass:At-structure}, \ref{ass:sdae-structure}, \ref{ass:coeff}, \ref{ass:FG-in-H} and \ref{ass:time_reg_merged} hold, and let $\theta \in (0,1]$, $h \in (0,h_{0}]$. Then there exist constants $C > 0$ and $q \geq 0$ such that for all $t \in [0,T-h]$ and $u \in \R^{d}$, 
      \begin{equation}\label{eq:verify-one-step-w-theta}
            \left|\E\big[w(t+h,U(t,u;t+h))\big] 
            - \E\big[w(t+h,\widetilde{u}(t,u;t+h))\big]\right|
            \leq
            C(1+|u|^{q})h^{2}.
      \end{equation}
\end{proposition}

\begin{proof}
      From the triangle inequality, Propositions \ref{prop:frozen_step_weak_error} and  \ref{lem:implicit_correction_weak}, it follows that
      \begin{align*}
            &~\left|\E\big[w(t+h,U(t,u;t+h))\big]
            -\E\big[w(t+h,\widetilde{u}(t,u;t+h))\big]\right| 
            \\\leq&~
            \left|\E\big[w(t+h,U(t,u;t+h))\big]
            - \E\bigl[w(t+h,\bar{u}(t,u;t+h))\bigr]\right| 
            \\&~
            + \left|\E\big[w(t+h,\bar{u}(t,u;t+h))\bigr]
            - \E\bigl[w(t+h,\widetilde{u}(t,u;t+h))\bigr]\right|
            \\\leq&~
            C(1+|u|^{q})h^{2},
      \end{align*}
      which proves \eqref{eq:verify-one-step-w-theta} and finishes the proof.
\end{proof}

Thus, together with the moment estimates in Lemma \ref{lem:moment_Y}, the one-step weak local error condition in Theorem~\ref{thm:weak_convergence_sdae} is verified with $p = 1$ for the stochastic theta method.

\subsubsection{Weak order one}
We now conclude the weak convergence order of the stochastic theta method. The result follows from the abstract weak convergence theorem with $p = 1$, together with the constraint preservation, moment bounds, and one-step weak local error estimate established above.

\begin{theorem}\label{thm:weak_order_one_theta}
      Suppose that Assumptions \ref{ass:At-structure}, \ref{ass:sdae-structure}, \ref{ass:coeff}, \ref{ass:FG-in-H} and \ref{ass:time_reg_merged} hold, and let $\theta \in (0,1]$, $h \in (0,h_{0}]$. Then there exists a constant $C > 0$, independent of $h$ and $N$, such that for all $\varphi \in C_{\mathrm{pol}}^{4}(\R^{d};\R)$,
      \begin{align*}
            \big|\E\big[\varphi(X_{T})\big] - \E\big[\varphi(Y_{N})\big]\big|
            \leq
            Ch.         
      \end{align*}
\end{theorem}

\begin{proof}
      Conditions \textup{(i)} and \textup{(ii)} in Theorem \ref{thm:weak_convergence_sdae} are ensured by Propositions \ref{thm:wp_invariance} and \ref{lem:moment_Y}, respectlvely. Besides, Proposition \ref{thm:wp_invariance} implies that $Y_{n} = u_{n} + \widehat{V}(t_n,u_n)$ for all $n = 0,1,\cdots,N$, which together with the one-step recursion $Y_{n+1} = Y(t_{n},Y_{n};t_{n+1})$ yields
      \begin{align*}
            Y_{n+1} = Y(t_{n},u_{n} + \widehat{V}(t_n,u_n);t_{n+1}), \quad n = 0,1,\cdots,N-1.
      \end{align*}
      Applying $P$ to both sides yields
      \begin{align*}
            u_{n+1} = PY_{n+1} = PY(t_{n},u_{n} + \widehat{V}(t_n,u_n);t_{n+1})
            = \widetilde{u}(t_n,u_n;t_{n+1}), \quad n = 0,1,\cdots,N-1,
      \end{align*} 
      which verifies Condition \textup{(iii)} in Theorem \ref{thm:weak_convergence_sdae}. Moreover, Proposition~\ref{prop:verify_one_step_w_theta} validates the one-step weak local error condition, namely Condition \textup{(iv)}, with $p = 1$. Therefore all assumptions of Theorem \ref{thm:weak_convergence_sdae} are satisfied. Applying that theorem with $p = 1$ finally completes the proof.
\end{proof}

\section{Numerical experiments}\label{sec:experiments}
This section presents a series of numerical experiments for two concrete examples to illustrate the previous theoretical results, including the structure-preserving property and the weak convergence order of the stochastic theta method. Before proceeding, we explain how \(Y_{n+1}\in\mathcal M_{t_{n+1}}\) is obtained from a given value \(Y_n\in\mathcal M_{t_n}\) through the stochastic theta method \eqref{eq:theta-scheme}. Since $A_{t_n}$ is singular, one cannot solve the method by inverting $A_{t_n}$. Instead, $Y_{n+1}$ is obtained as the solution of the nonlinear system
\begin{equation*}
      \mathcal{R}_n(y)
      =
      A_{t_n}y - A_{t_n}Y_{n}
      -
      h\big((1-\theta)F(t_n,Y_n) + \theta F(t_{n+1},y)\big) 
      -
      G(t_n,Y_n)\Delta W_n, \quad y \in \R^{d}.
\end{equation*}
Then $Y_{n+1}$ is obtained by solving $\mathcal{R}_{n}(Y_{n+1}) = 0$ via the Newton iteration with precision $10^{-5}$. Moreover, the solution $Y_{n+1}$ obtained in this way automatically satisfies the algebraic constraint at the new time level, i.e., $Y_{n+1} \in \mathcal{M}_{t_{n+1}}$. Indeed, applying $R$ to $\mathcal{R}_{n}(Y_{n+1}) = 0$ and using $RA_{t_n} = 0$, $RG(t_n,Y_n) = 0$, and $RF(t_n,Y_n) = 0$, we obtain $h\theta RF(t_{n+1},Y_{n+1}) = 0$. It follows from $\theta > 0$ that $RF(t_{n+1},Y_{n+1}) = 0$, i.e., $Y_{n+1} \in \mathcal M_{t_{n+1}}$.

In what follows, all expectation-type quantities in the numerical experiments are approximated
by the Monte Carlo method with $M = 10^{4}$ different Brownian sample paths. Given a stepsize $h > 0$ and  $T = Nh$ for some $N \in \N$, let $Y_{n}^{(h,m)} \approx X(t_{n}), n = 0, 1, \cdots, N$ denote the numerical solutions generated by the stochastic theta method along the $m$-th Brownian sample path. To illustrate the constraint-preserving property, we report the Monte Carlo root-mean-square constraint residual at each time level as follows
\begin{equation*}
      \operatorname{Res}(t_n)
      :=
      \left(\frac{1}{M}\sum_{m=1}^{M}
      \big|RF(t_n,Y_n^{(h,m)})\big|^2\right)^{1/2},
      n = 0, 1, \cdots, N.
\end{equation*}
We next describe how the weak errors are computed at the final time $T$. Since the exact solution is generally unavailable, we use a reference approximation  computed with a sufficiently small stepsize $h_{\mathrm{ref}}$ along the $m$-th Brownian sample path, denoted by $Y_{T}^{(h_{\mathrm{ref}},m)}$. For any test function $\varphi \in C_{\mathrm{pol}}^{4}(\R^{d};\R)$, the weak error at time $T$ is approximated by the following Monte Carlo estimator
\begin{equation*}
      \operatorname{Err}(T)
      :=
      \left|\frac{1}{M}\sum_{m=1}^{M}
      \left(\varphi\big(Y_{T}^{(h_{\mathrm{ref}},m)}\big)
      -
      \varphi\big(Y_T^{(h,m)}\big)\right)\right|,
\end{equation*}
where the observable $\varphi$ is chosen from the following four scalar test functionals
\begin{equation*}
      \varphi_{1}(x) := \sum_{i=1}^{d}x_{i},\quad
      \varphi_{2}(x) := \sum_{i=1}^{d}x_{i}^{2},\quad
      \varphi_{3}(x) := \cos\left(\sum_{i=1}^{d}x_{i}\right),\quad
      \varphi_{4}(x) := \frac{1}{1+\sum_{i=1}^{d}x_{i}^{2}}
\end{equation*}
for any $x = (x_{1},\cdots,x_{d})^{\top} \in \R^{d}$.

\begin{example}\label{ex:smib}
{\rm{
      Consider a reduced single-machine infinite-bus (SMIB) model with stochastic load 
      \begin{equation}\label{eq:SMIB}
      \left\{\begin{aligned}
            d\delta_t
            &= (\omega_t-\omega_s)\,dt, 
            \\
            d\omega_t
            &= \frac{\omega_s}{2H}
            \bigl(T_m-T_{e,t}-D(\omega_t-\omega_s)\bigr)\,dt, 
            \\
            d\eta_t
            &= -\alpha \eta_t\,dt+\beta\,dW_t, 
            \\
            0
            &= \kappa\sin\delta_t+P_L(1+\rho\eta_t)-T_{e,t} 
      \end{aligned}\right.
      \end{equation}
      with the initial value $\eta_{0} = 0$, $\omega_{0} = \omega_{s}$, $T_{e,0} = T_{m}$ and $\delta_{0} = \arcsin((0.8-0.3)/1.5825)$, where $\{W_t\}_{t\in[0,T]}$ is a standard real-valued Brownian motion; see \cite{milano2013systematic} for the precise physical interpretation of $\delta_{t}$, $\omega_{t}$, $\eta_{t}$ and $T_{e,t}$. Here we take $T_m=0.8$, $P_L=0.3$, $H=3.0$, $\omega_s=120\pi\,\mathrm{rad/s}$, $\kappa=1.5825$, $D=1.0610\times10^{-3}$, $\alpha=1.0$, $\beta=0.2$ and $\rho=0.3$; see \cite[Problem 8.4]{SauerPaiChow2017}. Letting $X_{t} := (\delta_t, \omega_t, \eta_t, T_{e,t})^{\top}$ enables us to rewrite \eqref{eq:SMIB} as an SDAE $A\,dX_t = F(X_t)\,dt + G\,dW_t$ with
      \begin{equation*}
            A=
            \begin{pmatrix}
                  1&0&0&0\\
                  0&2H/\omega_s&0&0\\
                  0&0&1&0\\
                  0&0&0&0
            \end{pmatrix},
            \quad
            F(x)=
            \begin{pmatrix}
                  x_2-\omega_s\\
                  T_m-x_4-D(x_2-\omega_s)\\
                  -\alpha x_3\\
                  \kappa\sin x_1+P_L(1+\rho x_3)-x_4
            \end{pmatrix},
            \quad
            G=
            \begin{pmatrix}
                  0\\0\\ \beta\\0
            \end{pmatrix}.
      \end{equation*}
      This example is a constant-matrix special case of the SDAEs framework considered above. The diffusion is compatible with the algebraic constraint, the algebraic equation $\kappa\sin x_1 + P_L(1 + \rho x_3) - x_4 = 0$ is uniquely solvable with respect to $x_4$, and the coefficients are smooth with bounded derivatives of all required orders. Therefore, Assumptions \ref{ass:At-structure}, \ref{ass:sdae-structure}, \ref{ass:coeff}, \ref{ass:FG-in-H} and \ref{ass:time_reg_merged} are satisfied, and the weak convergence result in Theorem \ref{thm:weak_order_one_theta} applies.

      \begin{figure}[!htbp]
      \begin{center}
            \subfigure[$T=50$]
            {\includegraphics[width = 4cm, height = 3cm]
            {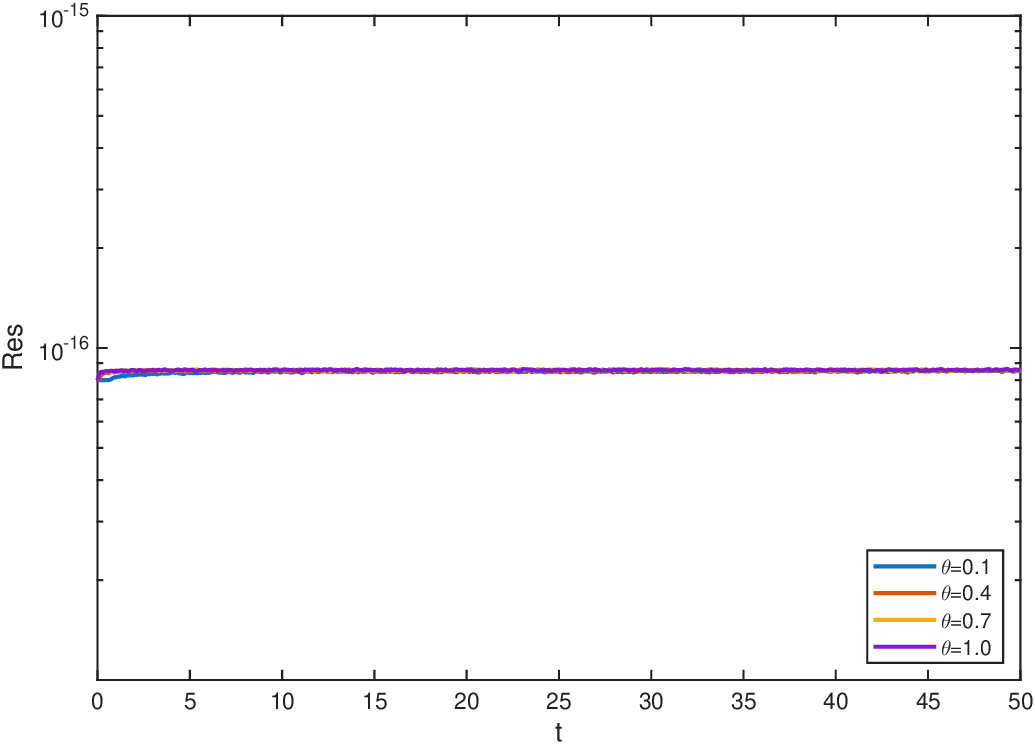}}
            \quad
            \subfigure[$T=100$]
            {\includegraphics[width = 4cm, height = 3cm]
            {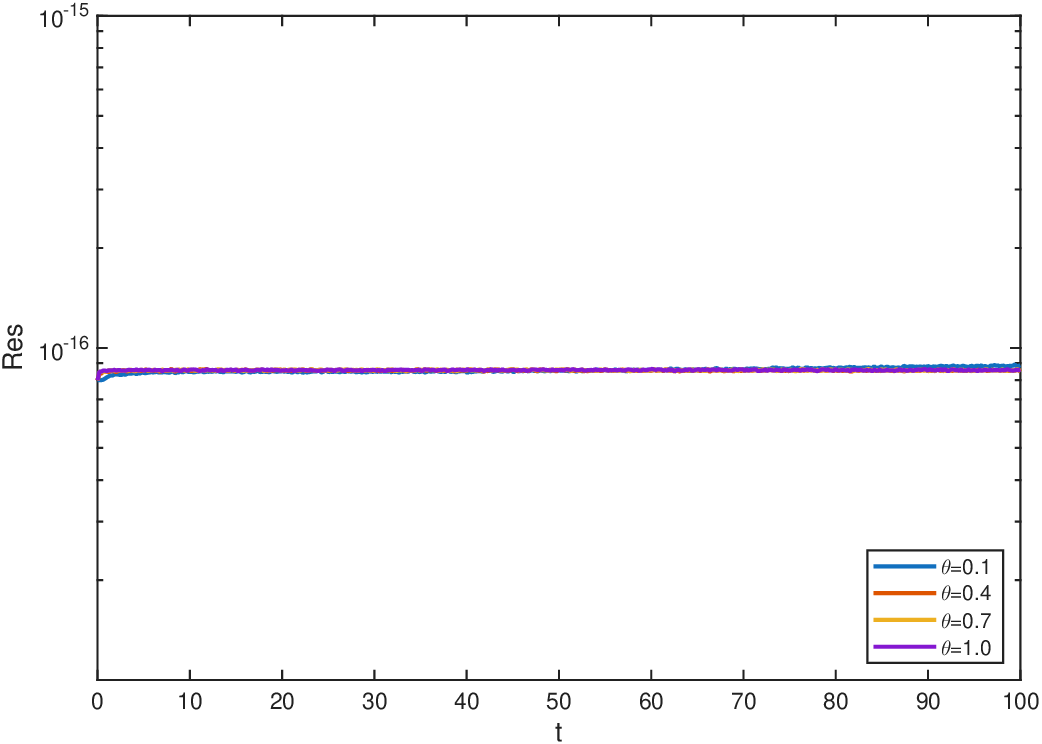}}
            \quad
            \subfigure[$\theta=0.1$]
            {\includegraphics[width = 4cm, height = 3cm]
            {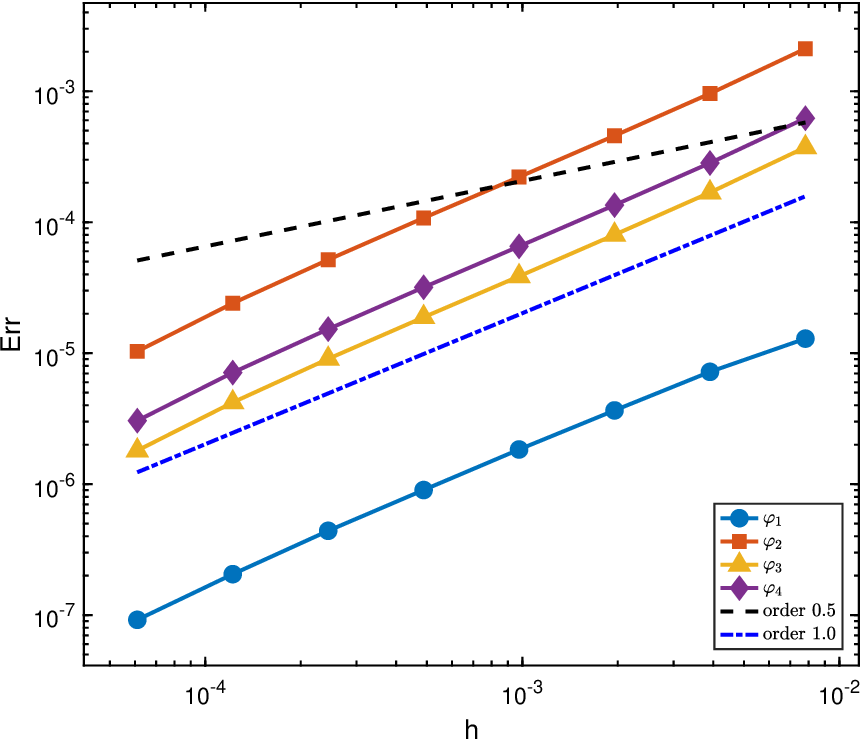}}
            \\
            \subfigure[$\theta=0.4$]
            {\includegraphics[width = 4cm, height = 3cm]
            {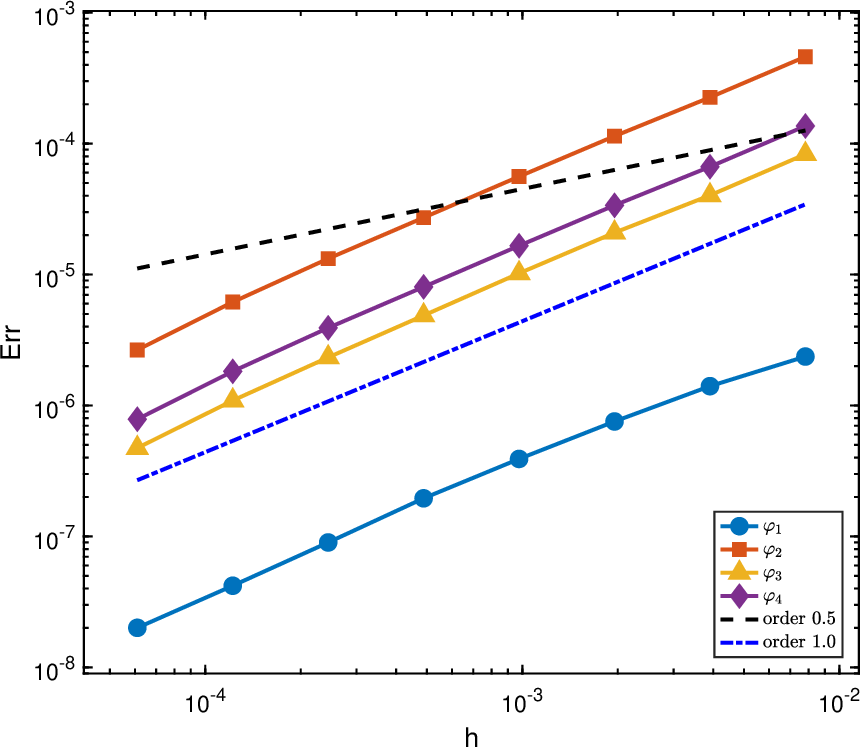}}
            \quad
            \subfigure[$\theta=0.7$]
            {\includegraphics[width = 4cm, height = 3cm]
            {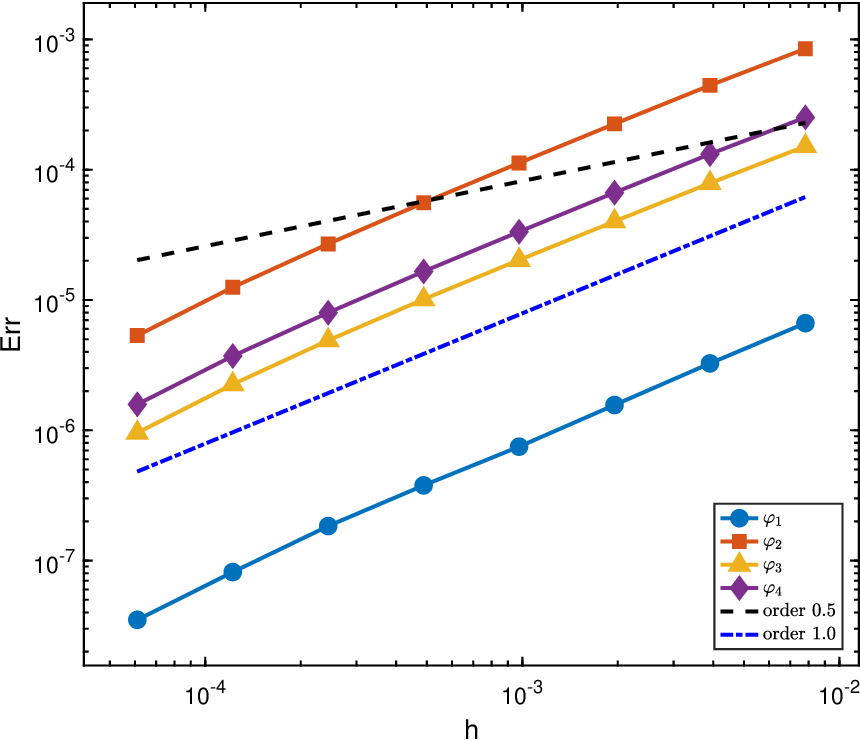}}
            \quad
            \subfigure[$\theta=1.0$]
            {\includegraphics[width = 4cm, height = 3cm]
            {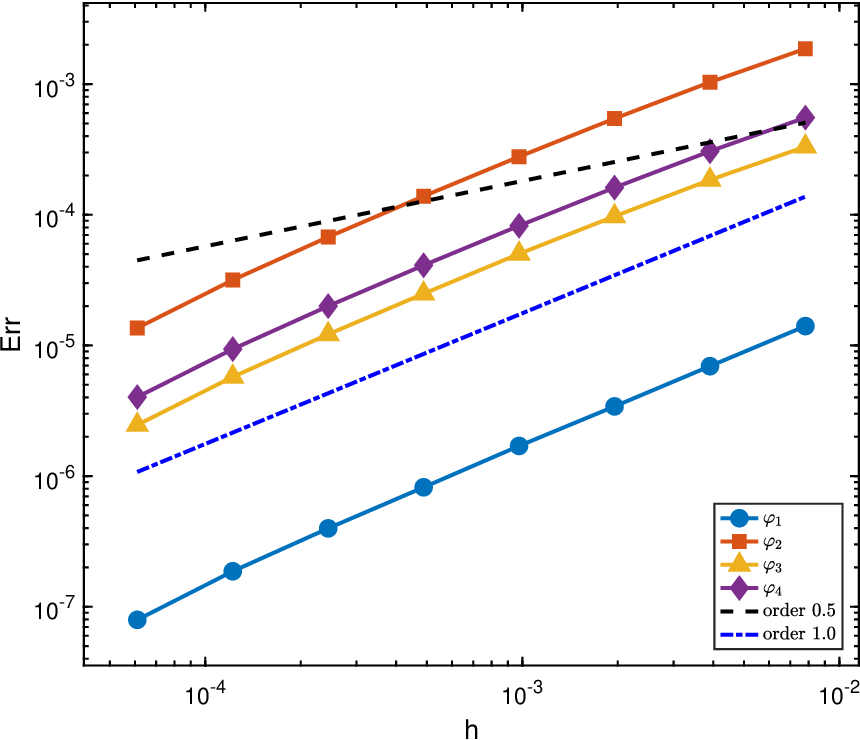}}
            \caption{structure-preserving simulations in (a)--(b) and weak convergence orders in (c)--(f)}
            \label{fig:example1}
      \end{center}
      \end{figure}

      Figure \ref{fig:example1} illustrates the constraint-preserving property and the weak convergence behaviour of the stochastic theta method. Panels (a) and (b) show the Monte Carlo root-mean-square residuals of the algebraic constraint for $T = 50$ and $T = 100$, respectively, with $h = 2^{-10}$. For all tested values of $\theta$, the residuals remain close to the round-off level throughout the simulations, confirming that the proposed method preserves the algebraic constraint over both short and relatively long time intervals. Panels (c)--(f) of Figure~\ref{fig:example1} report the weak errors for $T = 1$, $h_{\mathrm{ref}} = 2^{-16}$ and $h = 2^{-14},2^{-13},\ldots,2^{-7}$. For $\theta = 0.1, 0.4, 0.7, 1.0$, the weak errors corresponding to the four test functionals decrease approximately linearly with respect to $h$ on the log-log scale. The observed slopes agree with the reference line of order one, in accordance with Theorem \ref{thm:weak_order_one_theta}. These results indicate that the stochastic theta method achieves weak order one while preserving the algebraic constraint structure of the SDAE.
}}
\end{example}

\begin{example}\label{ex:td_singular}
{\rm{
      Let $\{W_t\}_{t \in [0,T]}$ be a standard real-valued Brownian motion and consider       \begin{equation}\label{eq:ex3}
            A_t\,dX_t = F(t,X_t)\,dt + G(t,X_t)\, dW_t,
            \qquad t \in [0,T]
      \end{equation}
      with the initial value $X_{0} = (\gamma, 1 + \eta\tanh(\gamma) + \chi)^{\top}$, where $A_{t} = 
      \big(1 + \frac12\sin t\big)\begin{pmatrix} 1&0 \\ 0&0 \end{pmatrix}$ and
      \begin{equation*}
            F(t,x)=
            \binom{\lambda x_1+\beta\tanh(x_1)+\mu\sin t}
            {x_2-1-\eta\tanh(x_1)-\chi\cos t},
            \quad
            G(t,x)=
            r(1+\nu\cos t)\frac{x_2}{1+x_2^2}
            \binom{1}{0}.
      \end{equation*}
      In the simulation, we take $\gamma = 1$, $\eta = 0.8$, $\chi = 0.3$, $\lambda = 0.4$, $\beta = 1.2$, $\mu = 0.2$, $r = 0.9$ and $\nu = 0.2$. For this example, the singular matrix $A_{t}$ varies in time while its image and kernel remain fixed. Moreover, the diffusion is compatible with the algebraic constraint, and the algebraic equation $x_2-1-\eta\tanh(x_1)-\chi\cos t = 0$ is explicitly solvable with respect to the algebraic variable $x_2$. The coefficients are smooth and satisfy the required growth and regularity conditions. Hence Assumptions \ref{ass:At-structure}, \ref{ass:sdae-structure}, \ref{ass:coeff}, \ref{ass:FG-in-H} and \ref{ass:time_reg_merged} are fulfilled, and Theorem \ref{thm:weak_order_one_theta} thus holds. Under the same numerical setting as in Example \ref{ex:smib}, Figure \ref{fig:example2} again shows that the stochastic theta method preserves the algebraic constraint and achieves weak convergence of order one.

      \begin{figure}[!htbp]
      \begin{center}
            \subfigure[$T=50$]
            {\includegraphics[width = 4cm, height = 3cm]
            {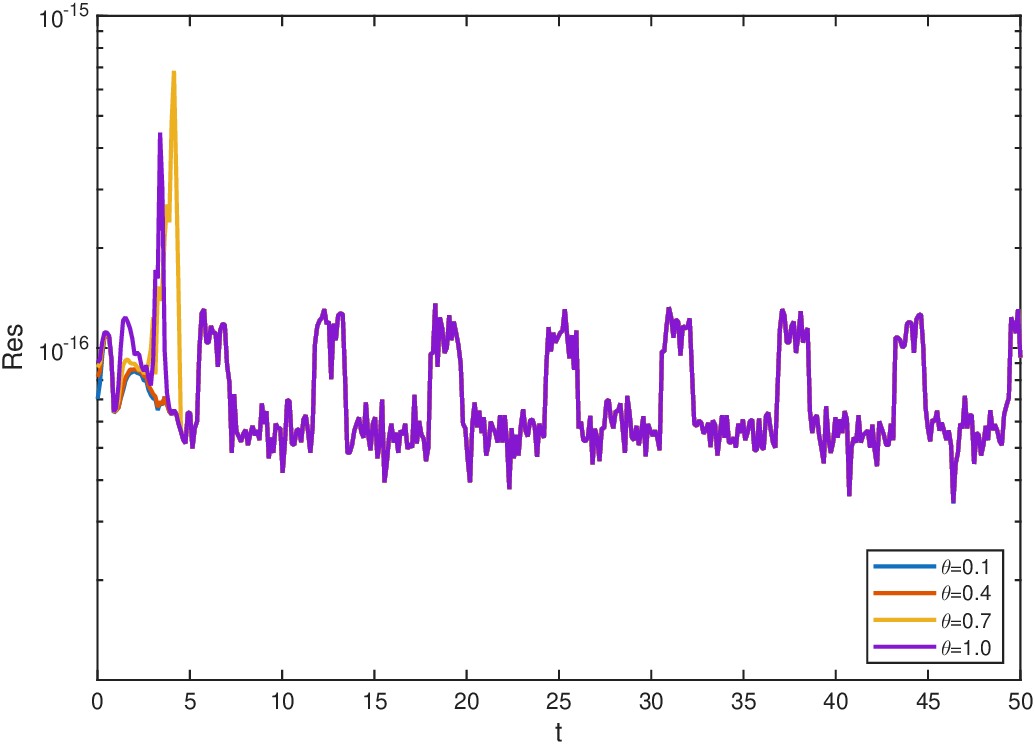}}
            \quad
            \subfigure[$T=100$]
            {\includegraphics[width = 4cm, height = 3cm]
            {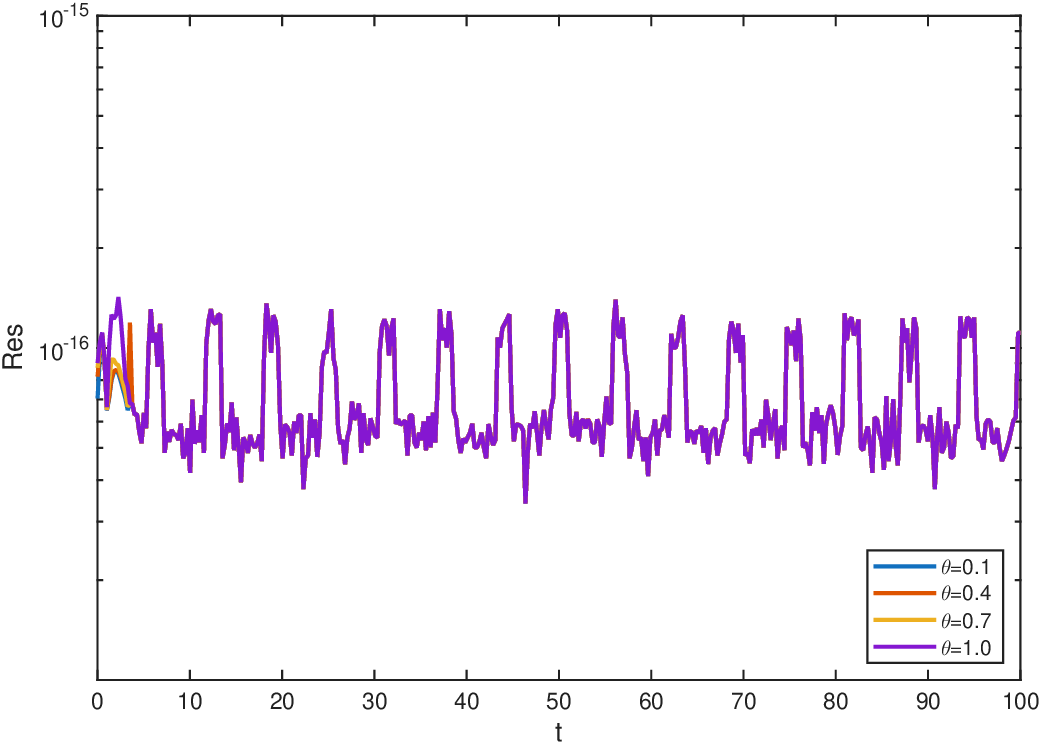}}
            \quad
            \subfigure[$\theta=0.1$]
            {\includegraphics[width = 4cm, height = 3cm]
            {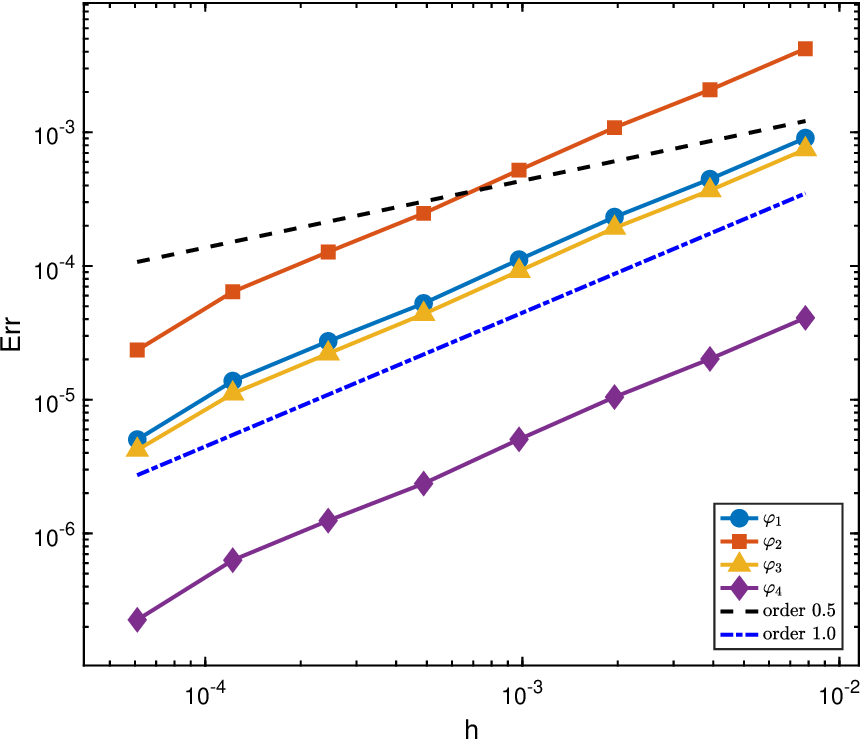}}
            \\
            \subfigure[$\theta=0.4$]
            {\includegraphics[width = 4cm, height = 3cm]
            {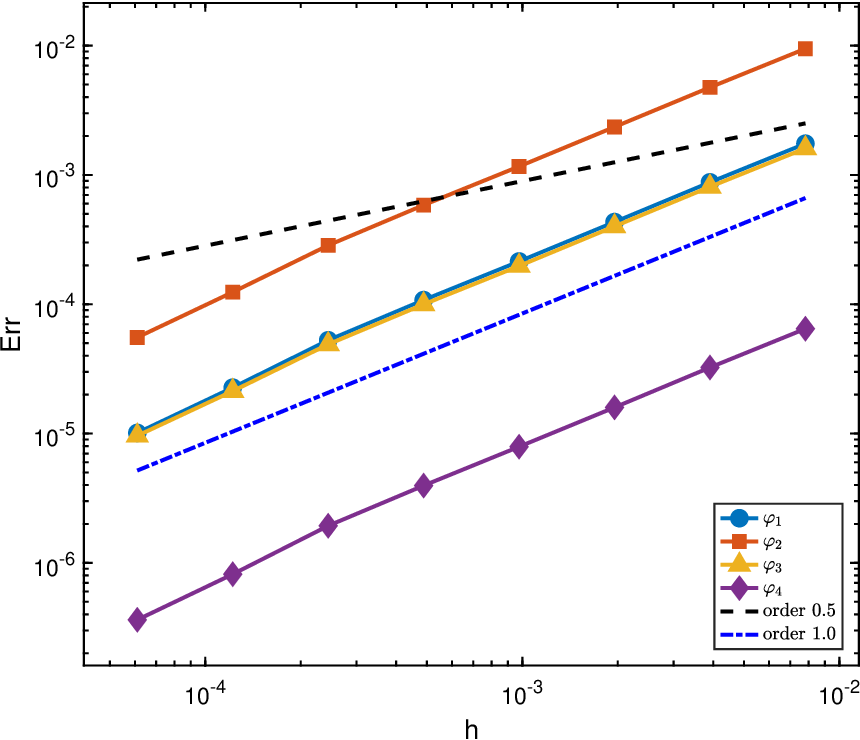}}
            \quad
            \subfigure[$\theta=0.7$]
            {\includegraphics[width = 4cm, height = 3cm]
            {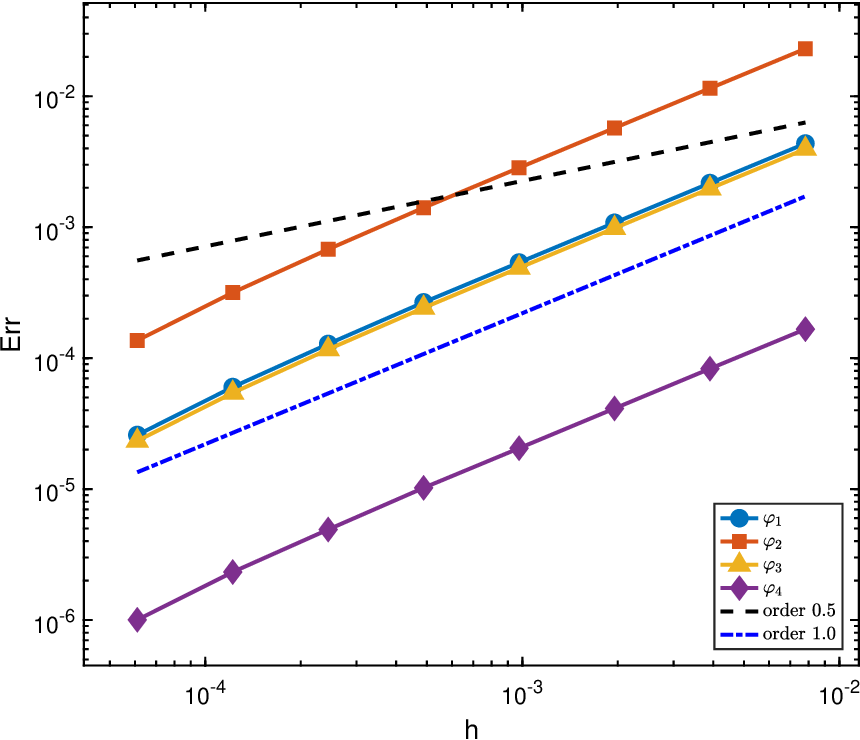}}
            \quad
            \subfigure[$\theta=1.0$]
            {\includegraphics[width = 4cm, height = 3cm]
            {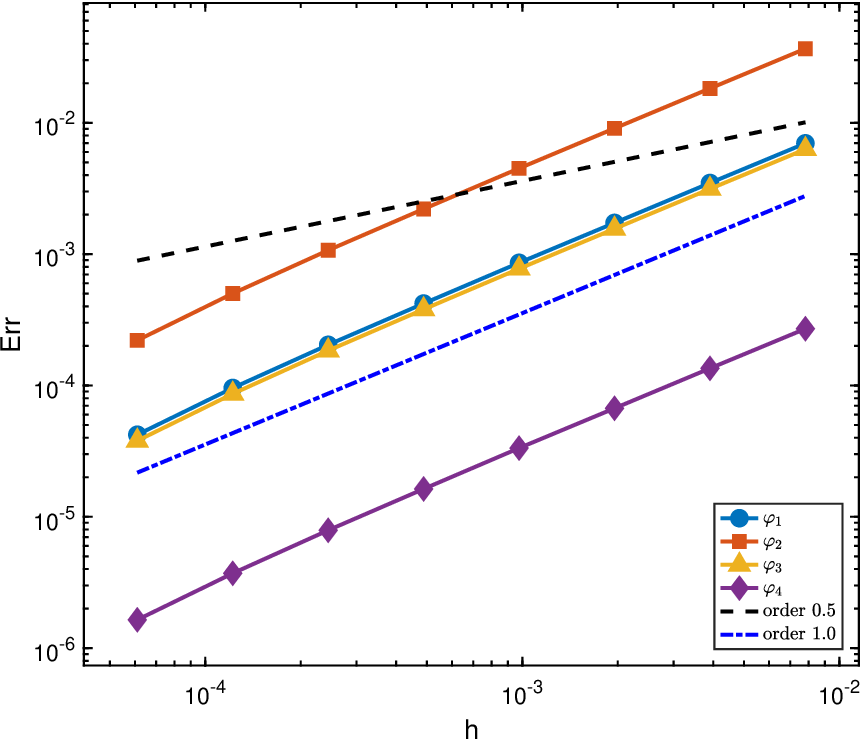}}
            \caption{structure-preserving simulations in (a)--(b) and weak convergence orders in (c)--(f)}
            \label{fig:example2}
      \end{center}
      \end{figure}
}}
\end{example}

\section{Conclusion and future work}
In this paper, we investigated weak convergence of structure-preserving stochastic theta methods for index-$1$ SDAEs with time-dependent singular matrices and a fixed differential-algebraic splitting. Based on the algebraic-differential decomposition, we established an abstract weak convergence theorem for constraint-preserving one-step approximations and applied it to the stochastic theta method with $\theta \in (0,1]$. Under global Lipschitz, linear growth, and suitable smoothness assumptions, we proved
that the method is well posed, preserves the algebraic constraints at each time level, and converges with weak order one. Numerical experiments confirmed the constraint-preserving property and the theoretical convergence order. These results provide a weak convergence framework that links the SDAE-level constraint-preserving discretization with the weak error analysis of the induced differential component.

Several directions deserve further study. One natural problem is to extend the present analysis, still under the fixed differential-algebraic splitting, to SDAEs with non-globally Lipschitz coefficients. This would require new moment estimates and local weak error arguments, and may lead to tamed, truncated, or
fully implicit structure-preserving methods. Another important direction is to treat more general time-dependent singular matrices for which the differential and algebraic subspaces are themselves time-dependent. In that case, the projectors vary with time and additional terms may appear in the reduced
dynamics, making both constraint preservation and weak convergence analysis more delicate.

\bibliographystyle{plain}	
\bibliography{refs}
\end{document}